\documentclass{amsart}
\usepackage{amsmath,amssymb}
\usepackage{latexsym,amsthm,epsf,graphicx,graphics,psfrag,amscd,subfigure}
\usepackage[all]{xy}
\theoremstyle{plain}
\newtheorem{theorem}{Theorem}
\newtheorem{lemma}[theorem]{Lemma}
\newtheorem{proposition}[theorem]{Proposition}
\newtheorem{corollary}[theorem]{Corollary}
\newtheorem{claim}[theorem]{Claim}
\newtheorem{conjecture}[theorem]{Conjecture}

\newtheorem*{problem*}{Problem}
\newtheorem*{conjecture*}{Conjecture}
\newtheorem*{proposition*}{Proposition}
\newtheorem*{corollary*}{Corollary}
\newtheorem*{no label}{}

\theoremstyle{definition}
\newtheorem{definition}[theorem]{Definition}

\theoremstyle{remark}
\newtheorem{obs}[theorem]{Remark}

\numberwithin{theorem}{section}

\def\cal#1{\mathcal{#1}}
\def\Id{\textrm{Id}}

\def\inter#1#2#3{{#1}\bullet({#2},{#3})}

\def\figurefontsize{9}

\begin{document}

\title{Tightness and efficiency of irreducible automorphisms of handlebodies II}
\author{Leonardo Navarro Carvalho}
\address{Departamento de Matem\'atica Aplicada, IME---Universidade Federal Fluminense, Niter\'oi, RJ, Brazil.}
\email{leonardoc@id.uff.br}

\begin{abstract}
We finish proving that an irreducible automorphism $f$ of a handlebody is efficient if, and only if, a certain standard pair of dual $f$--invariant laminations have the geometric tightness property. In a previous paper it was proved that this tightness property implies efficiency. We now prove the converse.
\end{abstract}

%\primaryclass{57M99}

%\secondaryclass{57N37,37E99}


\maketitle

\section{Introduction}\label{S:intro}

\subsection{The problem and its context}\label{SS:intro:context}

The problem of classifying automorphisms of compact $3$--manifolds, up to isotopy, is the general motivation of this
paper. It is part of an ongoing joint project with Ulrich Oertel aiming at such a classification
\cite{UO:Autos,LC:tightness,LCUO:Classification,UO:MCG_compression}, in a sense similar to that of Nielsen-Thurston's
classification of automorphisms of compact surfaces. Here an {\em automorphism} means a self-homeomorphism (or
diffeomorphism, PL--isomorphism, depending on the category considered). In this
paper we are concerned with the case of automorphisms of handlebodies. In fact we are concerned only with {\em
irreducible}  automorphisms of handlebodies, also known as {\em generic} automorphisms, see Definition \ref{D:irreducible} or any of
\cite{UO:Autos,LC:tightness,LCUO:Classification}.  The current paper may be regarded as a follow-up to \cite{LC:tightness}.

Before stating our main result we describe its context. Given a handlebody $H$ and an orientation preserving
automorphism $f\colon H\to H$ we consider its isotopy class $[f]$, which we may refer to as $[f]\colon H\to H$.  Oertel
classifies these isotopy classes (``mapping classes'') as 1) {\em periodic}, 2) {\em reducible} or 3) {\em irreducible}
\cite{UO:Autos}, see also Subsection \ref{SS:intro:terminology} (Oertel refers to the third case as {\em generic}).  If an isotopy class of an automorphism $f\colon H\to H$ is irreducible (periodic, reducible), we will also say that $f$ is irreducible (periodic, reducible).

Given an irreducible mapping class $[f]$ of a handlebody $H$ our goal is to find a representative $f\colon H\to H$ which
is ``the best'', in some meaningful sense. Oertel \cite{UO:Autos} shows that such an automorphism will have many
similarities with a surface pseudo-Anosov diffeomorphism: he constructs a dual pair of $f$--invariant and mutually transverse measured
laminations of $\textrm{Int}(H)$ --- a $2$--dimensional $(\Lambda,\mu)$ and $1$--dimensional $(\Omega,\nu)$, the leaves of the first one being open discs, the last one with isolated
singularities.\footnote{The need for singularities comes from regarding $\Omega\subseteq \mathring{H}$
embedded. In fact Oertel's construction yields more: a non-singular abstract lamination $\Omega'$ and immersion
$i\colon \Omega'\to\mathring{H}$, $i(\Omega')=\Omega$. For the purpose of this paper all the needed information can be
extracted from its image $\Omega$.} Moreover, these laminations ``fill'' $H$ (in the sense of capturing its topology)
and there exists a scalar $\lambda>1$ such that $f(\Lambda,\mu)=(\Lambda,\lambda\mu)$,
$f(\Omega,\nu)=(\Omega,\lambda^{-1}\nu)$.

The pair $(\Lambda,\mu)$ and $(\Omega,\nu)$ of $f$--invariant measured laminations  --- and the associated {\em growth
rate} $\lambda>1$ ---
are not at all unique. When we propose a quest for ``the best'' representative or ``a best" representative $f\colon H\to H$, we expect to
characterize it
in terms of properties of its invariant laminations and growth rate. The growth $\lambda$ can be regarded as a
measure of the complexity of a representative,  so a best representative should realize the minimal
$\lambda$. Such a representative realizing minimal growth for a pair of invariant measured laminations $(\Lambda,\mu)$, $(\Omega,\nu)$, which always exists, is called {\em efficient}. The problem we address in this paper is:

\begin{problem*} Characterize efficiency of an irreducible handlebody automorphism $f$ in terms of geometric properties of the associated $f$--invariant laminations $(\Lambda,\mu)$, $(\Omega,\nu)$.
\end{problem*}

Oertel identifies a condition on the laminations which he proves to be necessary for efficiency. We say that the pair of
laminations $(\Lambda,\Omega)$ has the {\em incompressibility property} when the inclusion $(\Lambda-\Omega)\to
(H-\Omega)$ is $\pi_1$--injective on each leaf (of $\Lambda-\Omega$). While efficiency implies that the laminations have
this property \cite{UO:Autos} the converse is not true \cite{LC:tightness}. A handlebody analogue to the ``no
back-tracking'' condition of Bestvina and Handel \cite{BH:Tracks,BH:Surfaces} is easily seen to imply efficiency but it is not
realizable on every mapping class. In \cite{LC:tightness} we introduced the notion of ``tight'' invariant
laminations. In the following we assume that $(\Lambda,\mu)$, $(\Omega,\nu)$ are the invariant laminations for $f\colon
H\to H$ obtained according to \cite{UO:Autos}. If $\Delta\subseteq H$ is an immersed surface transverse to $\Omega$ for the moment we denote by $\nu(\Delta)$ the $\nu$--measure of $\Delta$ (eventually in the paper, due to the technical context, we will denote $\nu(\Delta)$ by $\inter{\Delta}{\Omega}{\nu}$).

\begin{definition}\label{D:tight}
The pair of measured laminations $\big((\Lambda,\mu),(\Omega,\nu)\big)$ is {\em tight} if given any leaf $L\in\Lambda$
(which is an open disc) and simple closed curve $\gamma\subseteq L$ the disc $\Delta'\subseteq L$ bounded by $\gamma$
has the property that $\nu(\Delta')\leq\nu(\Delta)$ for any $\Delta\subseteq H$ transverse to $\Omega$ with
$\partial\Delta=\gamma$. We may omit the measures and say that $(\Lambda,\Omega)$ is tight, or abusing notation even
more, that $\Lambda$, or $\Omega$, is tight.

By extension we may say that a representative $f$ is {\em tight} when it has $f$--invariant pair of laminations which is tight.

When $(\Lambda,\Omega)$ is not tight a disc $\Delta$ as above such that $\nu(\Delta')>\nu(\Delta)$ is called a {\em tightening disc}.
\end{definition}

Tightness is weaker than ``no back-tracking'' and also implies efficiency \cite{LC:tightness}. It is also stronger than
the incompressibility condition. It remains to verify that the tightness condition is realizable on every mapping class
to prove the following conjecture, posed in \cite{LC:tightness}:

\begin{conjecture}\label{Cj:tightness}[Tightness]
Tightness characterizes efficiency. More precisely, an irreducible automorphism $f\colon H\to H$ is efficient if and
only if the associated pair of $f$--invariant laminations is tight.
\end{conjecture}

In this paper we fill this gap, proving the theorem below and thus asserting the conjecture.

\begin{theorem}\label{Th:tightness}
Let $[f]\colon H\to H$ be an irreducible mapping class. If $f\colon H\to H$ is an efficient representative then it is
tight. In particular tightness is realizable in every irreducible mapping class.
\end{theorem}

We recall that in \cite{LC:tightness} we proved a partial version of the theorem, with some other technical hypotheses. We also proved that tightness yields nice properties:

\begin{corollary*}[3.22 of \cite{LC:tightness}]
If $f$ is tight then the minimal growth rate in the class of any power $\lambda_{\textrm{min}}([f^n])$ is $(\lambda(f))^n$.
\end{corollary*}

\begin{corollary*}[3.23 of \cite{LC:tightness}]
If $f$ is tight then the growth rate $\lambda$ (which is minimal) is less than or equal to the growth rate $\partial\lambda$ of its restriction $\partial f\colon \partial H\to \partial H$ to the
boundary $\partial H$ (which is pseudo-Anosov).
\end{corollary*}

These two, together with Theorem \ref{Th:tightness} then yield immediately:

\begin{corollary}\label{C:powers}
If a representative $f\colon H\to H$ is efficient then so is any power $f^n\colon H\to H$.
\end{corollary}

\begin{corollary}\label{C:restriction}
If a representative $f\colon H\to H$ is efficient than its growth rate $\lambda$ is less than or equal to the growth rate $\partial\lambda$ of its restriction $\partial f\colon \partial H\to \partial H$ to the
boundary.
\end{corollary}

Before sketching the strategy for proving Theorem \ref{Th:tightness} we recall some background constructions and terminology in the next section.

I deeply thank Ulrich Oertel for his encouragement, comments, ideas and suggestions, some coming from a very thorough reading of a previous version.

%%%%%%%%%%%%%%%%%%%%%%%%%%%%%%%%%%%%%%%%%%%%%%%%%%%%%%%%%%%%%%%%%%%%%%%%%%%%%%%%%%%%%%%%%%
%%%%%%%%%%%%%%%%%%%%%%%%%%%%%%%%%%%%%%%%%%%%%%%%%%%%%%%%%%%%%%%%%%%%%%%%%%%%%%%%%%%%%%%%%%

\subsection{Technical terminology and notation}\label{SS:intro:terminology}

We refer the reader to \cite{UO:Autos} and \cite{LC:tightness} for details on the constructions described in this subsection. Although the original construction of \cite{UO:Autos} is more complete the point of view of \cite{LC:tightness} is already adapted for our purposes, so we shall try to be consistent with its notation and terminology.

Recall that a $3$--manifold is a {\em compression body} \cite{FB:CompressionBody}, \cite{UO:Autos,LC:tightness} if it
can be cut-open along a finite collection of disjoint discs to yield a product $F\times I$, where $S$ is a compact
surface (not necessarily connected) and $I=[0,1]$ is the interval, and possibly some $3$--ball components. Note that a
handlebody is a special type of compression body. We are concerned with ``irreducible'' automorphisms of
handlebodies:

\begin{definition}\label{D:irreducible}
Let $H$ be a handlebody and $f\colon H\to H$ an automorphism. A {\em closed reducing surface} for the mapping class
$[f]$ is an embedded closed surface $S\neq S^2$, $\emptyset$ which is $[f]$--invariant (i.e., $f(S)$ is isotopic to $S$)
and separating $H$ into a handlebody and a compression body. We say that $[f]$ is {\em irreducible} (or {\em generic},
as in \cite{UO:Autos}) if
\begin{enumerate}
    \item there is no closed reducing surface for $[f]$, and
    \item the restriction $\partial f=f|_{\partial H}$ is isotopic to a pseudo-Anosov diffeomorphism.
\end{enumerate}
We may say that the actual automorphism $f\colon H\to H$ is {\em irreducible} when its mapping class $[f]$ is.
\end{definition}

From now on we assume that $[f]\colon H\to H$ is an irreducible mapping class of a handlebody $H$.

Oertel \cite{UO:Autos} makes the following construction for a representative $f$ of $[f]$. It starts with any
representative $f\colon H\to H$ and isotopes it until it has the desired properties. All these intermediate
automorphisms will be labeled $f$.

Choose an embedded ``concentric'' handlebody $H_0\subseteq H$. This means that $\overline{H-H_0}\simeq\partial H\times
I$. Now isotope $f$ so that it is ``outward expanding and exhausting'' with respect to $H_0$, meaning that
$H_0\subseteq\textrm{Int}(f(H_0))$, that $\bigcup_{i\in\mathbb{Z}} f^i(H_0)=\textrm{Int}(H)$, and that $\bigcap_{i\in\mathbb{Z}} f^i(H_0)$ has empty interior, yielding
\[
\cdots\subseteq H_{-1}\subseteq H_{0}\subseteq H_{1}\subseteq \cdots \subseteq H_{i}\subseteq H_{i+1}\subseteq\cdots
\]

Now consider $\mathcal{E}=\cal{E}_0=\{E_1,\dots,E_k\}$ a {\em system of essential discs} in $H_0$, briefly a {\em disc
system}, meaning it consists of properly embedded and pairwise disjoint discs $(E_i,\partial E_i)\subseteq
(H_0,\partial H_0)$ which are not boundary parallel. In this paper such a system will often have the {\em completeness}
property of cutting $H_0$ open into a union os balls.  We may abuse notation and also use  $\cal{E}_0$ to denote the union $\bigcup_{1\leq i\leq k} E_i$.

Such a complete disc system $\cal{E}_0$ determines a ``handle decomposition''
of $H_0$ where the $1$--handles are disjoint product neighborhoods $E_i\times I$ of the discs $E_i$ and the
$0$--handles are the components of the closure of the complement of the $1$--handles.
We label this handle decomposition $\mathcal{H}_0$, consisting of the $1$--handles $\mathcal{H}_0^1$ and $0$--handles $\mathcal{H}_0^0$.
We regard the product structure $E_i\times I$ of the $1$--handles in $\mathcal{H}_0^1$ as part of the decomposition. Both product foliations of each 1-handle $E_i\times I$ are important: each disc of the product foliation
$E_i\times\{p\}$ is called a {\em dual disc} of the $1$--handle; the interval leaves $\{q\}\times I$ of the transverse foliation are called {\em $I$--fibers}.

Let $H_1=f(H_0)\supseteq H_0$ and in it the disc system $\cal{E}_1=f(\cal{E}_0)$. Similarly we have the corresponding
$\cal{H}_1=f(\cal{H}_0)$. We isotope $f$ further so that the $1$--handles of $\cal{H}_0$ intersect those of $\cal{H}_1$
in a manner {\em compatible} with their product structures (see Figure \ref{F:admissibility}): that means that 1) $1$--handles of $\cal{H}_1$ intersect $H_0$ in $1$--handles of $\cal{H}_0$ with 2) dual discs of $\cal{H}_1$ intersecting $H_0$ in dual discs of $\cal{H}_0$ and 3) $I$--fibers of $\cal{H}_0$ intersecting $1$--handles of $\cal{H}_1$ in $I$--fibers. Such a handle decomposition is said to be {\em admissible}. In particular $\mathcal{E}_1\cap H_0=f(\mathcal{E}_0)\cap H_0$ consists of dual discs of $\mathcal{H}_0$ and we also deem $\mathcal{E}_0$ {\em admissible}. A consequence of admissibility for the $0$--handles is $\mathcal{H}_0^0\subseteq\mathcal{H}_1^0=f(\mathcal{H}_0^0)$.

\begin{figure}[ht]
\centering
\psfrag{H1}{\fontsize{\figurefontsize}{12}$1\textrm{--handle of } \mathcal{H}_1$}
\psfrag{H0}{\fontsize{\figurefontsize}{12}$1\textrm{--handles of } \mathcal{H}_0$}
\includegraphics[scale=0.5]{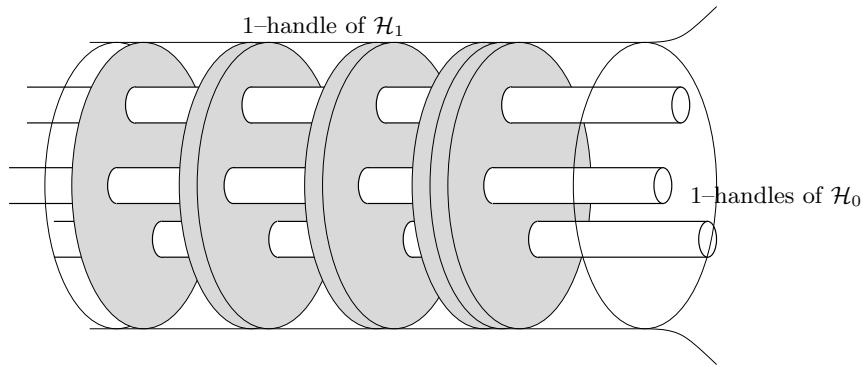}
\caption{\small A $1$--handle of $\mathcal{H}_1$ (short and wide), laminated by parallel dual discs (in dark), intersecting the $1$--handles of $\mathcal{H}_0$ (light and long and thin) in an admissible way.}\label{F:admissibility}
\end{figure}

Now for each $j\in\mathbb{Z}$ we may consider $H_j=f^j(H_0)$ and, as before, the corresponding $\cal{E}_j=f^j(\mathcal{E}_0)$ and $\cal{H}_j=f^j(\mathcal{H}_0)$.
We say that an embedded surface $S\subseteq H_n$ is {\em transverse to $\mathcal{H}_j$} if $S$ intersects $H_j$ in its $1$--handles and transverse to the $I$--fibers (of the $1$--handles of $\cal{H}_j$).\footnote{Here the definition is slightly more general than that from \cite{LC:tightness}, which requires $S\cap H_j$ to consist of dual discs for transversality to $\mathcal{H}_j$. An ambient isotopy supported in a neighborhood of the $1$--handles and preserving its $I$--fibers makes a surface transverse to $\mathcal{H}_j$ as defined in this current paper into one as in \cite{LC:tightness}.} For brevity we may also refer to such a surface as a {\em $j$--transverse surface}. From a dual point of view,  this amounts to saying that $S\cap H_j$ is {\em carried} by $\cal{E}_j$. In this case if $E\in\cal{E}_j$ and $D\subseteq S\cap H_j$ consists of components contained in the $1$--handle corresponding to $E$ then we also say that $D$ is {\em carried by} $E$. It is clear from admissibility that if $i<j$ and $S$ is $j$--transverse (i.e. $S\cap H_j$ is carried by $\cal{E}_j$) then $S$ is also $i$--transverse (i.e. $S\cap H_i$ is carried by $\cal{E}_i$).

Recall the enumeration $\mathcal{E}=\mathcal{E}_0=\{E_1,\dots,E_k\}$ and associate  to $\mathcal{E}$ the following ``incidence matrix'' $M$. For each
pair $1\leq i,j\leq k$ let the entry $M_{(i,j)}$ be the number of components of $f(E_i)\cap H_0$ which are carried by
$E_j$. We call it the {\em incidence matrix of $\mathcal{E}$ with respect to $f$}. We may label
$M=M(\mathcal{E})=M(f,\mathcal{E})$, depending on whether it is convenient to specify the representative or disc system. If $[f]$ is irreducible we can obtain a representative $f$ and a complete and admissible $\mathcal{E}$ such that the corresponding matrix is {\em
irreducible}, as defined in Subsection \ref{SS:intro:matrices} below. Therefore $M$ has an eigenvalue
$\lambda=\lambda(f)=\lambda(M)>1$ and a corresponding $\lambda$--eigenvector $\hat\nu\in\mathbb{R}^k$ with positive entries,
$M\hat\nu=\lambda \hat\nu$.

We regard the entries of the vector $\hat\nu=(\hat\nu_1,\dots,\hat\nu_k)$ as transverse measures on the foliations of $1$--handles of $\cal{H}_0$ by $I$--fibers: define in the $1$--handle corresponding to a disc $E_i$ a measure $\eta_i$ transverse to the foliation by $I$--fibers with total measure $\hat\nu_i$. We represent that as a pair $(\mathcal{H}_0,\hat\nu)$, called the {\em standard weighted handle decomposition} of $H_0$. If $S$ is a $0$--transverse surface we write its ``$\hat\nu$--measure'' as $\inter{S}{\mathcal{H}_0}{\hat\nu}=\eta_1(S)+\eta_2(S)+\dots+\eta_k(S)$, and call it {\em the weighted intersection of $S$ with $(\mathcal{H}_0,\hat\nu)$}. Is is clear form the construction that the $i$--entry $\hat\nu_i=\inter{E_i}{\mathcal{H}_0}{\hat\nu}$. In fact
\begin{equation}\label{eq:intersection-0}
\inter{S}{\mathcal{H}_0}{\hat\nu}=\sum_{D_l\subseteq S\cap H_0} \hat\nu_{i(l)},
\end{equation}
where the sum is taken over components $D_l\subseteq S\cap H_0$ and $\hat\nu_{i(l)}$ is the weight (total measure of $\eta_{i(l)}$) assigned to the $1$--handle to which $D_l$ is transverse, recalling $\hat\nu_{i(l)}$ is the $i(l)$--entry of the vector $\hat\nu$. Therefore $\inter{S}{\mathcal{H}_0}{\hat\nu}$ may be regarded as a sort of ``counting measure with weights'', by counting components of $S\cap H_0$ with weights given by the corresponding entries of $\hat\nu$.

Through $f$ we obtain a {\em standard weighted handle decomposition} on every $H_j$ as the pair $(\mathcal{H}_j,\lambda^j\hat\nu)$. For brevity we write $\hat\nu^j=\lambda^j\hat\nu$, therefore $(\mathcal{H}_j,\hat\nu^j)$ also denotes the standard weighted handle decomposition on $H_j$. Similarly, if $S$ is a $j$--transverse surface we consider the measure, or {\em weighted intersection}
$\inter{S}{\mathcal{H}_j}{\hat\nu^j}=\lambda^j\cdot(\inter{f^{-j}(S)}{\mathcal{H}_0}{\hat\nu})$. From this definition and (\ref{eq:intersection-0}) above follows
\begin{equation}\label{eq:intersection-n}
\inter{S}{\mathcal{H}_j}{\hat\nu^j}=\sum_{D_l\subseteq S\cap H_j} \lambda^j\hat\nu_{i(l)},
\end{equation}
hence, as before, $\inter{S}{\mathcal{H}_j}{\hat\nu^j}$ should be regarded as a counting measure with weights, by counting components of $S\cap H_j$ with weights given by the corresponding entries of $\lambda^j\hat\nu$.

The $\lambda$--eigenvector property of $\hat\nu$ imply that, if $S$ is both $i$--transverse and $j$--transverse then:
\[
\inter{S}{\mathcal{H}_i}{\hat\nu^i}=\inter{S}{\mathcal{H}_j}{\hat\nu^j}.
\]

Consider the limit:
\[
 \Omega=\bigcap_{j\in\mathbb{Z}} H_j=\bigcap_{j\in\mathbb{Z}} f^j(H_0),
\]
which is a $1$--dimensional lamination (with isolated singularities, corresponding to the intersection of the $0$--handles of the $H_j$'s). Define a transverse measure $\nu$ on it in the following manner. If $S$ is transverse to $\Omega$ then it is transverse to some $\mathcal{H}_j$ and define
\[
\nu(S)=\inter{S}{\mathcal{H}_j}{\hat\nu^j},
\]
obtaining a measured lamination $(\Omega,\nu)$. Note that if $i<j$ then $S$ is also $i$--transverse and $\inter{S}{\mathcal{H}_i}{\hat\nu^i}=\inter{S}{\mathcal{H}_j}{\hat\nu^j}$, hence the definition of $\nu$ does not depend on $j$. We rather denote $\nu(S)=\inter{S}{\Omega}{\nu}$ and call it
the {\em weighted intersection} (or just the {\em intersection}) of $S$ with $(\Omega,\nu)$. We can regard $(\mathcal{H}_j,\hat\nu^j)$ as a sort of combinatorial approximation of $(\Omega,\nu)$.

From a dual point of view we obtain the transverse $2$--dimensional measured lamination $(\Lambda,\mu)$. We only outline its construction, that being quite analogous to that of $(\Omega,\nu)$. In addition, in this paper $(\Lambda,\mu)$ does play a secondary role compared with $(\Omega,\nu)$.

Consider the transpose $N=M^T$ of $M=M(f,\cal{E})$ and $\lambda$--eigenvector $\hat\mu$. It corresponds to a measure transverse to the foliation of the $1$--handles by discs, which through powers $f^j$ of $f$ yields measures $\lambda^{-j}\hat\mu$ on the $1$--handles of all $\cal{H}_j$'s. Consider
\[
 \Lambda=\bigcup_{j\in\mathbb{Z}} f^j(\bigcap_{i\geq 0} \mathcal{H}_i^1),
\]
and note that $\bigcap_{i\geq 0} \mathcal{H}_i^1\subseteq H_0$, viewed as a closed set foliated by discs, is carried by $\mathcal{E}$ hence $\Lambda\cap H_j$ is
carried by $\mathcal{E}_j$ for any $j\in\mathbb{Z}$. In fact $\bigcap_{i\geq 0} \mathcal{H}_i^1=\Lambda\cap H_0$. The vectors $\lambda^{-j}\hat\mu$ then yields approximations to a transverse measure on $\Lambda$, which determine $(\Lambda,\mu)$: if $\alpha$ is an embedded compact arc transverse to $\Lambda$ then it is transverse to the foliation by discs of (the $1$--handles of) some $\mathcal{H}_j$, whose measure determines $\mu(\alpha)$.

The property that $M=M(f,\cal{E})$ is irreducible imply that both measured lamination $(\Omega,\nu)$, $(\Lambda,\mu)$ have full support.

\subsection{The strategy}\label{SS:intro:strategy}

The purpose of this subsection is to motivate our approach to proving Theorem \ref{Th:tightness}, to sketch the argument, and  to indicate the organization of the argument in the paper. There are no proofs or important definitions in the sketch.

We will prove Theorem \ref{Th:tightness} by contradiction: supposing that
the representative $f$, together with $f$--invariant lamination pair $(\Lambda,\Omega)$, is not tight, the goal is to obtain a more
efficient representative.  Let $\Delta$ be a tightening disc, as in Definition \ref{D:tight} above. Let $L\in\Lambda$ be the leaf such that $\gamma=\partial\Delta\subseteq L$ and $\Delta'\subseteq L$ be the disc bounded by $\gamma$. If $\Delta$ has certain special
properties one pre-composes $f$ with an ambient isotopy $\varphi$ which takes a laminated product neighborhood of
$\Delta'$ (an original $1$--handle of the handle decomposition) to a similar neighborhood of $\Delta$ (a new $1$--handle), taking
$\Delta'$ to $\Delta$. This ``tightening'' process of replacing $\Delta'$ with the ``more efficient'' $\Delta$ yields a
more efficient representative (perhaps after performing some standard simplifications). Therefore, if the original representative
$f$ was efficient it was necessarily also tight.

Of course the outline above hides most of the technical difficulties in the vague phrase ``if $\Delta$ has {\em certain
special properties} ...''. Both \cite{UO:Autos} and \cite{LC:tightness} deal with cases where the ``special
properties'' considered are too restrictive: in the first, essentially that $\inter{\Delta}{\Omega}{\nu}=\nu(\Delta)=0$; in the second,
essentially that $(\Delta,\partial\Delta)\subseteq(H_0,\partial H_0)$ and that $f(\Delta)\cap H_0$ is carried by
$\mathcal{E}=\mathcal{E}_0$.  These properties cannot be assumed in general, for an arbitrary lamination pair which is not tight.
We need some less restrictive properties which are achievable in full generality and sufficient for our purposes.

Recall from Section \ref{SS:intro:terminology} (and \cite{UO:Autos}, \cite{LC:tightness}) the construction of the pair $(\Lambda,\mu)$, $(\Omega,\nu)$ associated to $f$, involving a choice of admissible disc system $\mathcal{E}=\mathcal{E}_0=\{E_1,\dots,E_k\}\subseteq H_0\subseteq H$ and corresponding handle decomposition $\mathcal{H}_0$.
This yields through powers $f^j$ disc systems $\mathcal{E}_j\subseteq H_j=f^j(H_0)$.

Now we bring the tightening disc $\Delta$ back into the picture. We can assume 1) that $\Delta\subseteq H_0$ and that it is
parallel to a disc $\Delta'=E_1\in\mathcal{E}\subseteq H_0$; 2) that
$\Delta$ is disjoint from $\mathcal{E}$ so it can be added to the system, yielding a larger system \cite{LC:tightness}. Suppose
further that $f(\Delta)\cap H_0$ is also carried by $\mathcal{E}$. Then enlarging $\mathcal{E}'=\mathcal{E}\cup\{\Delta\}$ yields
$f(\mathcal{E}')\cap H_0$ still carried by $\mathcal{E}$.
If $\varphi$ is an ambient isotopy realizing the parallelism from $E_1$ to $\Delta$ then $f\circ\varphi(\mathcal{E}')\cap H_0$ is carried by $\mathcal{E}\subseteq\mathcal{E}'$. By passing to a subsystem, and possibly using other simplifications, this procedure
can be made to work \cite{LC:tightness}. The argument is not general because the assumption that $f(\Delta)\cap H_0$ is also carried by $\mathcal{E}$ is not general.

It is too much to expect that $f(\Delta)\cap H_0$ is also carried by $\mathcal{E}$ but we can consider something slightly more general: suppose the original disc system $\mathcal{E}$ can be enlarged to a system $\mathcal{E}'\supseteq \mathcal{E}$,
$\Delta\in\mathcal{E}'$ which is admissible, i.e. $f(\mathcal{E}')\cap H_0$ is carried by $\mathcal{E}'$ itself, possibly not by $\mathcal{E}$ as before. Still, if $\varphi$ is an ambient isotopy realizing the parallelism between $E_1$ and $\Delta$, then $f\circ\varphi$ is more efficient than $f$ (by choosing a suitable subsystem of $\mathcal{E}'$ and possibly some other adjusting isotopies). This will be proved in this paper, mainly in subsections \ref{S:proof:matrices} and \ref{S:proof:tracking}: Section \ref{S:proof:matrices} considers how to obtain the desired representative and its corresponding incidence matrix; and Section \ref{S:proof:tracking} shows that the representative obtained is indeed the desired one, i.e. more efficient than the original one, by carefully comparing the two corresponding incidence matrices.

So consider the task of enlarging $\mathcal{E}$ to an admissible system $\mathcal{E}'$ containing $\Delta$, possibly after some adjusting isotopies. That accounts to most of the work of the paper, culminating in Proposition \ref{P:stabilized} (in its statement $\mathcal{E}\cup\mathcal{D}$ plays the role of the enlargement $\mathcal{E}'$, and $\phi$ is one of the said adjusting isotopies).
One of the first obstacles to realizing it is the following. It may happen that when we enlarge $\mathcal{E}\subseteq H_0$, adding new discs $D\subseteq H_0$ to it, we will obtain components of $f(D)\cap H_0$ which are not essential discs. We introduce the notions of ``incidence disc'' and ``incidence moves'' in Section \ref{S:incidence} in order to realize this fitness property: obtain enlargements $\mathcal{E}'\supseteq\mathcal{E}$ with the property that $f(\mathcal{E}')\cap H_0$ consists of essential discs. Being more precise, to obtain an ambient isotopy $h\colon H\to H$ such that $h\circ f(\mathcal{E}')\cap H_0$ consists of essential discs. This isotopy $h$ realizes the incidence move. It is also necessary to control how these moves affect the $\nu$--measure of surfaces.

So we can assume, roughly, that the image $f(\mathcal{E}')$ of a given enlargement of $\mathcal{E}\subseteq\mathcal{E}'\subseteq H_0$ intersect $H_0$ in essential discs. Our goal is to obtain an admissible enlargement $\mathcal{E}'$ containing the tightening disc $\Delta$. A naive but natural approach towards this goal is the following. Start enlarging $\mathcal{E}$ by adding $\Delta$ to it, obtaining an enlargement $\mathcal{D}$. Now consider $f(\mathcal{D})\cap H_0$, which we are now assuming consists of essential discs. If it is disjoint from $\mathcal{D}$ we can enlarge $\mathcal{D}$ with $f(\mathcal{D})\cap H_0$, obtaining $\mathcal{D}'$. If $f(\mathcal{D}')\cap H_0$ is disjoint from $\mathcal{D}'$ we can repeat the process, enlarge $\mathcal{D}'$ with $f(\mathcal{D}')\cap H_0$, obtaining $\mathcal{D}''$, and so one. This yields an increasing sequence of disc systems, so this process eventually ``stabilizes'', meaning not just that no new isotopy class of discs appear, but also in a sense that no newly added disc is ``more efficient'' than a previous one. By making suitable adjustments we achieve the goal: to obtain an enlargement $\mathcal{E}'$ of $\mathcal{E}$ containing the tightening disc $\Delta$ and such that $f(\mathcal{E}')\cap H_0$ carried by $\mathcal{E}'$.

The argument sketched in the paragraph above assumes that in the process of enlargement we obtain disc systems $\mathcal{D}'$'s with the property that $f(\mathcal{D}')\cap H_0$ is disjoint from $\mathcal{D}'$, a property which is not general. This is the second main obstacle to finish the proof of Theorem \ref{Th:tightness}. Note that this property, for this construction, means that powers $f^p(\Delta)$, $p\geq 1$ are disjoint from $\Delta$. That is so because the construction starts by adding $\Delta$ and taking powers, and $f^p(\Delta)$ is always disjoint from $\mathcal{E}$. Therefore the task is to change $f$ through an ambient isotopy $\phi$ so that $(\phi\circ f)^p(\Delta)\cap\Delta=\emptyset$, $p\geq 1$. In Section \ref{S:pushaway} we discuss the operation of ``pushing a surface away'' from $\Delta$, the main tool for realizing this property.

By combining the elements sketched so far we shall obtain ambient isotopies $\phi$, $\varphi\colon H\to H$ such that $\phi\circ f\circ\varphi\colon H\to H$ is more efficient than $f\colon H\to H$. The $\phi$ constructed will be essentially a composition of (realizations) of incidence moves and push-aways. And $\varphi$ realizing the ``tightening'' move, the parallelism between $E_1$ and $\Delta$ (where $E_1\in\mathcal{E}$ is the disc parallel to $\Delta$). This final construction is obtained in Subsection \ref{S:proof:matrices}.

For technical reasons, right in the beginning of the construction the tightening disc will be replaced by a suitable ``parallel tightening disc'', as in Definition \ref{D:parallel_tightening_disc}. Recall that we mentioned that subsections \ref{S:proof:matrices} and \ref{S:proof:tracking} deal with growth rates of incidence matrices, therefore we will discuss some technical auxiliary lemmas on such matrices in Subsection \ref{SS:intro:matrices} below.

\subsection{Matrices and eigenvalues}\label{SS:intro:matrices}

All matrices and vectors in this paper have non-negative entries. If $v$ is such a vector
we say it is {\em non-negative} and write $v\geq 0$, meaning all its entries $v_i\geq 0$. In fact given vectors of the same rank $u$, $v\in\mathbb{R}_+^n$ we write $u\leq v$ (respectively $u<v$) to mean that all pair of entries $u_i\leq v_i$ ($u_i<v_i$) for all $1\leq i\leq n$. If we regard $0$ as the null-vector then $v\geq 0$ does mean the entries are not negative.
Similarly, a {\em positive vector} $v>0$ is one with all $v_i>0$.
Matrices, in addition to being non-negative will have integer entries and will be square.

The reader should bear in mind that in the context of this paper matrices appear as incidence
matrices (mostly) of disc systems: they encode the combinatorics of the image of the system through $f$ intersecting
its corresponding $1$--handles.
We shall achieve efficiency by changing disc systems and isotoping
the automorphism. These changes affect the matrices, and in this subsection we consider the kinds
of matrix changes we shall face and how they affect efficiency.
Efficiency is measured by the corresponding eigenvalue (the smaller the more efficient)
and the following proposition is a standard tool for identifying when a move improves efficiency
(see e.g. \cite{ES:73,BH:Tracks}). Recall that a $n\times n$ non-negative matrix $N$ is {\em irreducible} if
given any pair $(i,j)$, $1\leq i\leq  n$, $1\leq j\leq n$ there is a power $N^p$ such that the corresponding entry
$(N^p)_{(i,j)}>0$. Such a matrix has a real eigenvalue $\lambda(N)\geq 1$ realizing the spectral radius with
corresponding
$\lambda$--eigenvector $v\in\mathbb{R}^n$, $v>0$ (the case $\lambda=1$ being that of transitive permutation matrices, which do not appear as incidence matrices of the irreducible handlebody automorphisms considered in this paper).
Recall that for vectors $u\leq v$ means $u_i\leq v_i$ for every $i$--entry.

\begin{proposition}\label{P:subinvariance}
Let $N$ be a non-negative and irreducible $n\times n$ matrix and
$v\in \mathbb{R}^n$ with $v_i\geq 0$ and $v\neq 0$. If
$$
Nv\leq \lambda v,
$$
then $\lambda(N)\leq\lambda$ and $v>0$. If, moreover,
$(Nv)_i<\lambda v_i$ for some $i$, then $\lambda(N)<\lambda$.
\end{proposition}

The following are elementary.

\begin{lemma}\label{L:subinvariance}
Suppose that $M$ is a non-negative $m\times m$ integer matrix, $\lambda>0$, $v\in \mathbb{R}^m$ with $v\geq 0$ and
$v\neq 0$.
\begin{enumerate}
\item[I-] If $Mv\leq \lambda v$ then $M^pv\leq \lambda^p v$ for any power $p\geq 1$.
\item[II-] If moreover $N$ is a non-negative $m\times m$ integer matrix with $Nv\leq Mv$ then $N^p v\leq M^p v\leq
\lambda^p v$.
\item[III-] If in addition to I, II for some $p\geq 1$ and for some $i$, $1\le i\le m$, the inequality $(M^pv)_i<\lambda^p v_i$  holds, then also $(N^pv)_i<\lambda^pv_i$.
\end{enumerate}
\end{lemma}
\begin{proof}
I and II follow easily by induction on $p$. And III follows from II: $(N^pv)_i\leq(M^pv)_i<\lambda^pv_i$.
\end{proof}

Given a non-negative $m\times m$ integer matrix $M=\{ M_{(i,j)}\}$ we call $N$ a {\em submatrix} of $M$ if $N$ is obtained
by
erasing a choice of columns of $M$ and their corresponding rows. More precisely,
let $I=\{i_1<i_2<\cdots<i_n\}\subseteq \{1,2,\dots,m\}$ be an {\em index subset}. Then $N=\{ M_{(i,j)}\,|\, i,j\in I \}$.
Given $I=\{i_1<i_2<\cdots<i_n\}$ and a vector $v\in\mathbb{R}^m$ we can ``extract'' the corresponding vector $\check
v=\check v(I)=(
v_{i_1},v_{i_2},\dots,v_{i_n})$.

We are especially interested in irreducible submatrices. Since
submatrices are obtained by erasing entries from a non-negative matrix we obtain:

\begin{corollary}\label{C:submatrix}
Let $M$, $\lambda>0$ and $v\in\mathbb{R}^m$ be as in the statement of Lemma \ref{L:subinvariance}. Also assume
$Mv\leq \lambda v$.
Let $N$ be a submatrix of $M$ corresponding to an index subset $I\subseteq \{1,2,\dots,m\}$. Obtain from $v$ the vector
$\check v=\check v(I)$ corresponding to $I$.
\begin{enumerate}
\item [1-] For every $i\in I$ and $p\geq 1$ holds $(N^p\check v)_i\leq (M^p v)_i\leq \lambda^p v_i=\lambda^p\check v_i$.

\item [2-] If, in addition, $N$ is irreducible and for some $p\geq 1$ and $i\in I$ the inequality $(M^p v)_{i}<\lambda^p v_{i}$ holds then $\lambda(N)<\lambda$.
\end{enumerate}
\end{corollary}
\begin{proof}

1- $(N^p\check v)_i\leq (M^pv)_i$ follows easily by induction on $p$. From Lemma \ref{L:subinvariance}
$(M^pv)_i\leq \lambda^p v_i$. And $\lambda^p v_i=\lambda^p\check v_i$ since $v_i=\check v_i$, $i\in I$.

2- By Proposition \ref{P:subinvariance} $\lambda(N)\leq\lambda$. We prove the statement by contradiction.  Supposing the inequality  $\lambda(N)<\lambda$ is false, we have $\lambda(N)=\lambda$.
Since $N\check v\leq \lambda\check v$, by the second statement in Proposition \ref{P:subinvariance} we obtain
that $N\check v=\lambda\check v$. Then $N^p\check v=\lambda^p\check v$ for any $p\geq 1$. By 1 above for any $i\in I$ holds $(N^p\check v)_i\leq (M^p v)_i$. Therefore for the power $p\geq 1$ and index $i\in I$ in the statement holds $(N^p\check v)_i\leq (M^p v)_i<\lambda^p v_i=\lambda^p\check v_i$, which gives the required contradiction.
\end{proof}

We shall also need (for proving Lemma \ref{L:first_move}):

\begin{lemma}\label{L:propagation_matrix}
Let $M$ be an $m\times m$ irreducible matrix. For every $1\leq j\leq m$ (determining a row)
consider $q=q(j)\geq 1$ the
smallest integer such that the first entry of the $j$--row of $M^q$ is not zero.
Suppose $\bar M$ is obtained from $M$
by changes only on the first row.
For any $2\leq j\leq m$ the $j$--row of $\bar M^p$, $1\leq p\leq q=q(j)$ equals that
of $M^p$ (in particular the first entry of the $j$--row of $\bar M^q$ is non-zero).
\end{lemma}
\begin{proof}
First note that $q(j)\in\mathbb{Z}$ is well-defined since $M$ is irreducible. We prove the lemma by induction on $q(j)$. For any $2\leq j\leq m$ with $q(j)=1$ the result is immediate. Suppose it holds for every $j$ with $q(j)=q\geq 1$ and let $j$ be such that $q(j)=q+1$.
Let $1\leq p\leq q+1$. Then $\bar M^p=\bar M^{p-1} \bar M$.
But since $p-1\leq q$, the induction hypothesis implies that the $j$--row of
$\bar M^{p-1}$ equals that of $M^{p-1}$.  Hence the first entry in the $j$--row is zero, and then the product on the right by $\bar M$ is not affected by changes on the first row.
Therefore the $j$--row of $\bar M^p$ equals that of $M^p$.
\end{proof}

\section{Incidence discs}\label{S:incidence}

Recall that we shall prove Theorem \ref{Th:tightness} by assuming that $((\Lambda,\mu),(\Omega,\nu))$ is not tight and obtaining a contradiction.  Consider the disc system $\mathcal{E}=\mathcal{E}_0\subseteq
H_0$ and handle decomposition $\mathcal{H}_0$ corresponding to the invariant laminations.  The goal is to change $\mathcal{E}$ (and hence $\mathcal{H}_0$) and $f$ in order to reduce the growth rate.
This will be achieved by first enlarging $\mathcal{E}$ and then extracting a subsystem. It will also be necessary to isotope $f$.

The discs $E$ of the system $\mathcal{E}\subseteq H_0$ have the special property that $f(E)\cap H_0$ consists of
essential discs. Moreover, $f(E)\cap H_0$ is carried by $\mathcal{E}$, since $\mathcal{E}$ is admissible. In the process of enlarging
$\mathcal{E}$ it is desirable to add discs which have at least the first property. That is: essential discs $D\subseteq
H_0$ such that $f(D)\cap H_0$ consists of essential discs.\footnote{Towards obtaining the second {\em admissibility} property we further need to consider product neighborhoods of the essential discs of $f(D)\cap H_0$ and regard them as new $1$--handles.} This yields the following generalization.

\begin{definition}\label{D:incidence_disc}
An essential disc $(D,\partial D)\subseteq (H_n,\partial H_n)$ is an {\em incidence disc} if:
 \begin{enumerate}
    \item $D\cap H_{n-1}$ consists of essential discs, and
    \item $D\cap\Lambda\subseteq\partial D$ ($D\cap\Lambda$ often being empty).
 \end{enumerate}
If moreover $D\cap(H_{n}-\mathring{H}_{n-1})$ is incompressible, then it is called an {\em incompressible incidence
disc}.

An incidence disc $D$ which is also {\em $i_0$--transverse} (i.e. transverse to $\mathcal{H}_{i_0}$) for some $i_0\leq n$ is, of course, an {\em $i_0$--transverse incidence disc}.

\end{definition}

\begin{obs}\label{Rmrk:incidence_disc}
It is clear from the definition that if $(D,\partial D)\subseteq (H_n,\partial H_n)$ is an incidence disc then
$D\cap (H_{n}-\mathring{H}_{n-1})$ is a planar surface with one boundary component in $\partial H_{n}$ and the others
in $\partial H_{n-1}$, these bounding essential discs in $H_{n-1}$.
\end{obs}

\subsection{Standard incidence moves}

We enlarge $\mathcal{E}$ by adding incidence discs to it. This involves a construction
which obtains an incidence disc from a more general essential disc through an ambient isotopy (Lemma
\ref{L:make_incidence_tool} below). The isotopy need not change $\mathcal{E}$. This yields the notion below. By an {\em ambient isotopy} we mean an isotopy $H\times I\to H$ of the identity, i.e. $(p,0)\mapsto p$ for all $p\in H$. Its {\em support} is the closure of the complement of the fixed set $\{p\in H\,|\, (p,t)\mapsto p\, \textrm{, for all $t\in I$}\, \}$.

\begin{definition}\label{D:framed-isotopy}
A {\em framed isotopy of $H_0$} (or simply a {\em framed isotopy}) is an ambient isotopy $H\times I\to H$ with support in $H_0$ and fixing $\mathcal{E}_0\subseteq H_0$. Automorphisms isotopic through a framed isotopy are said to be {\em framed isotopic} in $H_0$.
We may abuse notation and say that an automorphism $H\to H$ supported in $H_0$ and framed isotopic to the identity is a {\em framed isotopy} in $H_0$.
Similarly we say two sub-manifolds $S, S'\subseteq H_0$ are {\em framed isotopic} if there exists a framed isotopy taking $S$ to $S'$.

In fact by extension we may refer to an ambient isotopy $H\times I\to H$ with support in some $H_n$ and fixing $\mathcal{E}_n\subseteq H_n$ as a {\em framed isotopy of $H_n$}. Also abusing notation we may say two automorphisms $H\to H$ isotopic through a framed isotopy of $H_n$ are {\em framed isotopic in $H_n$}. But unless clearly specifying an $H_n$ as setting we will typically consider as framed isotopy those supported in $H_0$.
\end{definition}

In \cite{LC:tightness}, Lemma 3.13, we proved the lemma stated below, after some straightforward adaptation of terminology and setting. Recall from Subsection \ref{SS:intro:terminology} that $(\mathcal{H}_{i_0},\hat\nu^{i_0})$ is the standard weighted handle decomposition of $H_{i_0}$, with vector $\hat\nu^{i_0}=\lambda^{i_0}\hat\nu$.

\begin{lemma}\label{L:make_incidence_tool}
Let $(D,\partial D)\subseteq(H_n,\partial H_n)$ be an $i_0$--transverse disc, for some $i_0<n$.
If $D\cap\mathcal{E}_n\subseteq\partial D$ then $D$ is framed isotopic in $H_n$ to an $i_0$--transverse incidence disc $D'\subseteq H_n$ such that:
\[
\inter{D'}{\mathcal{H}_{i_0}}{\hat\nu^{i_0}}\leq\inter{D}{\mathcal{H}_{i_0}}{\hat\nu^{i_0}}.
\]
\end{lemma}

\begin{obs}\label{R:making_incidence}
 The proof of Lemma \ref{L:make_incidence_tool} consists of changing $D$ using three kinds of moves (or ``processes''
as in \cite{LC:tightness}). These are:
\begin{enumerate}
 \item If $D\cap\partial{H}_{n-1}$ is inessential in $\partial H_{n-1}$, bounding an innermost disc
$\Delta\subseteq\partial H_{n-1}$, {\em compress} $D$ along $\Delta$, i.e. surger $D$ on $\Delta$ then discard the resulting sphere;
 \item If $D\cap(H_{n}-\mathring{H}_{n-1})$ is compressible, compress it along a compressing disc $\Delta\subseteq(H_{n}-\mathring{H}_{n-1})$.
 \item If $D\cap(H_{n}-\mathring{H}_{n-1})$ contains a component which is $(\partial H_{n-1})$--parallel, isotope it using the parallelism.
\end{enumerate}

We call these {\em standard incidence moves}. They are all realizable by framed isotopies of $H_n$:
for the first two, by an isotopy along a ball, the third by and isotopy through a product. None alter $\partial D$. Framed isotopies of $H_n$ do not introduce intersections with $\mathcal{E}_n$. Moreover, these do not introduce intersection of $D$ with $H_{i_0}$. In
particular the weighted intersection $\inter{D}{\mathcal{H}_{i_0}}{\hat\nu^{i_0}}$ is not increased. To complete the proof of the
lemma one has to verify that a finite sequence of such moves yields a disc intersecting $H_{n-1}$ in essential discs.
\end{obs}

The following is an improvement of the lemma. A simpler version has already been exploited in \cite{LC:tightness}.

\begin{lemma}\label{L:make_incidence}
Let $\mathcal{D},\mathcal{S}\subseteq H_0$ be disjoint disc systems transverse to some $\mathcal{H}_{i_0}$, $i_0<0$,
both disjoint from
$\mathcal{E}_0$. Suppose also that $\mathcal{D}$ consists of incompressible incidence discs. There exists $h\colon
H\to H$ framed isotopic in $H_0$ to the identity and such that
\begin{enumerate}
    \item $h|_\mathcal{D}=\Id|_\mathcal{D}$ and $h(\mathcal{S})$ consists of $i_0$-transverse incompressible incidence discs with
        \[
        \inter{h(\mathcal{S})}{\mathcal{H}_{i_0}}{\hat\nu^{i_0}}\leq \inter{\mathcal{S}}{\mathcal{H}_{i_0}}{\hat\nu^{i_0}};
        \]
    \item if $G\subseteq H_0$ is an $i_0$--transverse embedded surface disjoint from $\mathcal{S}$ (not necessarily a
disc) then $h(G)$ is  ${i_0}$--transverse and intersection with $(\mathcal{H}_{i_0},\hat\nu)$ is not increased:
        \[
        \inter{h(G)}{\mathcal{H}_{i_0}}{\hat\nu^{i_0}}\leq \inter{G}{\mathcal{H}_{i_0}}{\hat\nu^{i_0}}.
        \]
\end{enumerate}
\end{lemma}
\begin{proof}
By induction it suffices to prove the result for $\mathcal{S}=\{ S\}$ consisting of a single disc. Indeed, suppose that
the result is true for any $\mathcal{S}^j=\{ S_1,\dots, S_j\}$ with $j\in\mathbb{Z}$, $j\geq 0$ components. Note that
the case $j=0$ means $\mathcal{S}^j=\emptyset$, for which the result is obviously true ($h=\Id$). Let
$\mathcal{S}^{j+1}=\{S_1,\dots,S_j,S_{j+1}\}$ have $j+1$ components and consider
$\mathcal{S}^j\subsetneq\mathcal{S}^{j+1}$. The induction hypothesis says (1) that there exists framed isotopy $h_j$
such that $h_j(\mathcal{D})=\mathcal{D}$, $h_j(\mathcal{S}^j)$ consisting of $i_0$--transverse incompressible incidence
discs with
\[
\inter{h_j(\mathcal{S}^j)}{\mathcal{H}_{i_0}}{\hat\nu^{i_0}}\leq \inter{\mathcal{S}^j}{\mathcal{H}_{i_0}}{\hat\nu^{i_0}}.
\]
Moreover, by (2) $h_j(S_{j+1})$ is an $i_0$--transverse essential disc with
$\inter{h_j(S_{j+1})}{\mathcal{H}_{i_0}}{\hat\nu^{i_0}}\leq \inter{S_{j+1}}{\mathcal{H}_{i_0}}{\hat\nu^{i_0}}$.

Therefore consider $\mathcal{D}'=\mathcal{D}\cup h_j(\mathcal{S}^j)$ and $\mathcal{S}'=\{ S_{j+1}\}$. If we prove the result
for such an $\mathcal{S}'$ with just one component (which is the induction step) then we obtain $h$ satisfying properties
(1) and (2) as in the statement. It is easy to check that $h_{j+1}=h\circ h_j$ has the desired properties for
$\mathcal{D}$ and $\mathcal{S}^{j+1}$.
We now prove the induction step, assuming $\mathcal{S}=\{ S\}$. We verify both properties, one at a time.

\medskip

(1) According to Lemma \ref{L:make_incidence_tool} (see also remark \ref{R:making_incidence}) there exists a sequence of
standard incidence moves in $H_0-\mathring{H}_{-1}$ so that $S$ intersects $H_{-1}$ in $i_0$--transverse essential discs
without increasing intersection with $(\mathcal{H}_{i_0},\hat\nu)$.
We now prove that each of these standard incidence moves $h$ may be chosen not to introduce intersections of $S$ with
$\mathcal{D}$ and $h|_\mathcal{D}=\Id_\mathcal{D}$.

The first two types of standard incidence moves are compressions, where the compressing disc $\Delta\subseteq
(H_{0}-\mathring{H}_{-1})$: for the first move $\Delta\subseteq\partial H_{-1}$, while for the second $\Delta$ is
contained in the interior of $(H_{0}-\mathring{H}_{-1})$.
We deal with the two cases as one. In both $\Delta\subseteq \mathring{H}_0$. Let $D\subseteq S$ be the disc bounded by
$\partial\Delta=\partial D$. Also $D\subseteq \mathring{H}_0$. Moreover $D\cup\Delta$ is an embedded sphere bounding a
ball $B\subseteq \mathring{H}_0$ with $B\cap\mathcal{E}=\emptyset$. The move $h$ may be chosen so that its support
$N(B)$ satisfies $N(B)\subseteq\mathring{H}_0$ and $N(B)\cap\mathcal{E}=\emptyset$. If $\mathcal{D}\cap B=\emptyset$ we
can also assume $\mathcal{D}\cap N(B)=\emptyset$, so that $h$ preserves $\mathcal{D}$.

Assume then that $\mathcal{D}\cap B\neq\emptyset$. Recall that $B\subseteq\mathring{H}_0$, hence $B\cap\partial
H_0=\emptyset$. Since on every component $C\in\mathcal{D}$,  $\emptyset\neq\partial C\subseteq\partial H_0$ we see that
$C\nsubseteq B$. Therefore $\mathcal{D}\cap\partial B\neq\emptyset$. But $\partial B=D\cup \Delta$ and $\mathcal{D}\cap
D\subseteq\mathcal{D}\cap S=\emptyset$, therefore $\mathcal{D}\cap\Delta\neq\emptyset$ consists of closed curves. Such a
curve $\gamma\subseteq\Delta\subseteq (H_{0}-\mathring{H}_{-1})$ is therefore inessential in $(H_{0}-\mathring{H}_{-1})$.
Since $\mathcal{D}\cap (H_{0}-\mathring{H}_{-1})$ is incompressible we can change $\Delta$ by a framed isotopy (with support in $H_{0}-\mathring{H}_{-1}$) so that $\Delta\cap\mathcal{D}=\emptyset$. After repeating this finitely often we achieve $\mathcal{D}\cap B=\emptyset$. This proves that a standard incidence move $h$ of the first two
types may be chosen preserving $\mathcal{D}$.

The third type of standard incidence move consists of taking a component $F\subseteq S\cap(H_{0}-\mathring{H}_{-1})$
which is $(\partial H_{-1})$--parallel and isotoping it along a product $F\times [0,1]\simeq P\subseteq
(H_{0}-\mathring{H}_{-1})$,
where $F$ is identified with $F\times\{1\}$ and $F\times\{0\}\subseteq\partial H_i$. Call it $h$.

As before if $\mathcal{D}\cap P=\emptyset$ then $\mathcal{D}\cap N(P)=\emptyset$,
where $N(P)$ is the support of $h$. Then $h|_\mathcal{D}=\Id$.
We claim that $\mathcal{D}\cap P=\emptyset$ is always the case, so assume otherwise $\mathcal{D}\cap P\neq\emptyset$ by
contradiction.

It is easy to see that $P\cap\partial H_0=\emptyset$: indeed, $\partial F = F\cap(\partial H_0\cup\partial H_{-1})$. But
$\partial F\subseteq\partial H_{-1}$, hence $F\cap\partial H_0=\emptyset$. Therefore the same holds for the product:
$P\cap\partial H_0=\emptyset$.

As before, for every component $C\in\mathcal{D}$,  $\partial C\subseteq\partial H_0$. If $C\cap P\neq\emptyset$ for some
component, which is being assumed, then $C\cap\partial P\neq\emptyset$ and hence $\mathcal{D}\cap \partial
P\neq\emptyset$. Now $\partial P=F\times\{0,1\}$ and $\mathcal{D}\cap F\times\{1\}\subseteq\mathcal{D}\cap S=\emptyset$
by the  disjointness hypothesis. Therefore $\mathcal{D}\cap\partial P\subseteq F\times\{0\}\subseteq\partial H_{-1}$. If
$F'\subseteq\mathcal{D}\cap P$ is a component then $\partial F'\subseteq\partial H_{-1}$. But from $\mathcal{D}\cap
F\times\{1\}=\emptyset$ implies that $F'$ is also a component of $\mathcal{D}\cap(H_0-\mathring{H}_{-1})$. Hence $F'$
has a boundary component in $\partial H_0$ (see Remark \ref{Rmrk:incidence_disc}), a contradiction with $\partial
F'\subseteq\partial H_{-1}$. Therefore $\mathcal{D}\cap P=\emptyset$ proving that disjointness of $\mathcal{S}=\{S\}$
and $\mathcal{D}$ is preserved by the third standard incidence move.

This proves, using Lemma \ref{L:make_incidence_tool}, that if $\mathcal{S}=\{S\}$ there is a framed isotopy $h$ in $H_0$ (a composition of standard incidence moves) such that $h(\mathcal{D})=\mathcal{D}$ and $h(S)$ is an incompressible
$i_0$--transverse incidence disc satisfying:
\[
\inter{h(\mathcal{S})}{\mathcal{H}_{i_0}}{\hat\nu^{i_0}}\leq \inter{\mathcal{S}}{\mathcal{H}_{i_0}}{\hat\nu^{i_0}}.
\]

\medskip

(2) As before it suffices to prove the result for $h$ consisting of a single standard incidence move.

The third kind is the easiest, therefore first suppose $h$ realizes such a move. If $i_0=-1$ then
$S$ is already an incidence disc and hence $S\cap(H_0-\mathring{H}_{-1})$ consists of a planar surface with a boundary
component on $\partial H_0$ (see remark \ref{Rmrk:incidence_disc}). Therefore there is no $(\partial H_{-1})$--parallel
component, no incidence move of the third type and $h=\Id$. Assume then that $i_0\leq -2$. Recall notation from the
proof of (1) above: the support of $h$ is the neighborhood $N(P)$ of a product $P\subseteq H_{0}-\mathring{H}_{-1}$.
Therefore $N(P)\cap H_{i_0}=\emptyset$ and $h$ does not introduce intersection of any subset of $H$ with $H_{i_0}$. In
particular $h(G)\cap H_{i_0}=G\cap H_{i_0}$, where $G$ is the surface described in the statement.

As with the third kind of standard incidence move, each of the first two may be regarded as a framed isotopy along a product.
Indeed, recall from (1) the ball $B$ with $\partial B=D\cup \Delta$, where $D\subseteq S$ and $\Delta \subseteq
(H_{0}-\mathring{H}_{-1})$.
Regard $B$ with a product structure $B\simeq D^2\times I$ determining the parallelism
between $D$ and $\Delta$.  As in (1) if $G\cap B=\emptyset$ then $h(G)=G$ and then
$\inter{h(G)}{\mathcal{H}_{i_0}}{\hat\nu^{i_0}}=\inter{G}{\mathcal{H}_{i_0}}{\hat\nu^{i_0}}$.

Assume $G\cap B\neq\emptyset$. Consider neighborhoods $N(D)$, $N(\Delta)$ of $D$, $\Delta$ such that the support of $h$
is
$N(B)=N(D)\cup B\cup N(\Delta)$. Moreover that $h(B\cup N(\Delta))\subseteq N(\Delta)$ and $h^{-1}(B\cup N(D))\subseteq
N(D)$.
From the disjointness $G\cap D=\emptyset$ assume $G\cap N(D)=\emptyset$ and from $\Delta\subseteq
(H_{0}-\mathring{H}_{-1})$
that $N(\Delta)\cap H_{i_0}=\emptyset$. It follows that
if $p\in H_{i_0}\cap G\cap N(B)$ then $p\in B$ and hence $h(p)\in N(\Delta)$, which is disjoint from $H_{i_0}$.
In words, $h$ removes intersection of $G$ with $H_{i_0}$.

It remains to verify that $h$ does not introduce intersection of $G$ with $H_{i_0}$. Assume $h(p)\in H_{i_0}$, where $p\in
N(B)$
the support of $h$. From $N(\Delta)\cap H_{i_0}=\emptyset$ follows $h(p)\in B\cup N(D)$. By applying $h^{-1}$ we obtain
$p\in N(D)$.
But $G\cap N(D)=\emptyset$, hence $p\notin G$, proving that $h$ does not introduce intersections of $G$ with $H_{i_0}$.
Therefore
\[
 \inter{h(G)}{\mathcal{H}_{i_0}}{\hat\nu^{i_0}}\leq \inter{G}{\mathcal{H}_{i_0}}{\hat\nu^{i_0}}.
\]
\end{proof}
\begin{obs}\label{rm:make_incidence_improvement}
There is a sort of improvement of the lemma which is natural and easily obtained. If $\mathcal{D}$ has the property
that $\mathcal{D}\cap H_i$, $i_0\leq i\leq 0$ consists of incompressible incidence discs then the same property can be
realized for $h(\mathcal{S})$. The idea is to apply essentially the same procedure which we have described in $H_0-\mathring{H}_{-1}$ inductively
down the ``layers'' $H_{i+1}-\mathring{H}_i$.
\end{obs}

\subsection{Tightening discs}
In \cite{LC:tightness} the author introduces the following special type of tightening disc:

\begin{definition}\label{D:strong}
Let $(\Delta,\partial\Delta)\subseteq (H_n,\partial H_n\cap \Lambda)$, $n\geq 1$ be an incidence
disc and $\Delta'\subseteq\Lambda$ be such that
$\partial\Delta'=\partial\Delta$. We say that $\Delta$ is a {\em
strong tightening disc} if:
\begin{enumerate}
  \item it is transverse to $\mathcal{H}_0$ and for any $0\leq i\leq n$, $\Delta\cap H_i$ consists of incidence
  discs, and
  \item $\inter{\Delta}{\cal{H}_0}{\hat\nu}<\inter{\Delta'}{\cal{H}_0}{\hat\nu}$.
\end{enumerate}
\end{definition}

Of course it is more convenient to work with strong tightening discs then general ones, and there is no loss of generality in considering those:

\begin{proposition}[\cite{LC:tightness}]\label{P:equivalence}
There exists a tightening disc if and only if there exists a
strong tightening disc.
\end{proposition}

In this paper we also need a special type of tightening disc which has convenient properties but is also general enough.
Strictly speaking they are not tightening discs as in Definition \ref{D:tight}, but are naturally associated to these,
see Remark \ref{rm:parallel}. Recall $(\cal{H}_i,\hat\nu^i)$ is the standard weighted handle decomposition of $H_i$ (see
Subsection \ref{SS:intro:terminology}).

\begin{definition}\label{D:parallel_tightening_disc}
An incidence disc $(\Delta,\partial\Delta)\subseteq(H_n,\partial H_n)$ is a {\em parallel tightening disc} if $\Delta$
is disjoint from $\mathcal{E}_n$, parallel to a disc $E\in\mathcal{E}_n$ and
$\inter{\Delta}{\Omega}{\nu}<\inter{E}{\Omega}{\nu}$. If $\Delta$ is transverse to $\mathcal{H}_i$ we say that it is a
{\em parallel tightening disc with respect to $H_i$} or in short just an {\em $i$--tightening disc}.
\end{definition}

\begin{obs}\label{rm:parallel}
We mentioned that parallel tightening discs and strong tightening discs are naturally related. This is how to obtain a disc of a type from one of the other.

If $\Delta$ is a strong tightening disc one can obtain a $0$--tightening disc by essentially pushing it away from
$\Lambda$. More precisely, if $\Delta'\subseteq\Lambda$ is the disc bounded by $\partial\Delta$ then $\Delta'$ is a dual disc, parallel to some $E\in\mathcal{E}_n$. Use the product structure $E\times I$ of the corresponding $1$--handle to
push $\Delta$ along the intervals away from it, see Figure \ref{F:strong-parallel}. That yields a $0$--tightening disc $D$.

{\def\scalesize{0.3}

\begin{figure}[ht]
\centering
\mbox{\subfigure{\psfrag{Delta'}{\fontsize{\figurefontsize}{12}$\Delta'$}
\psfrag{Delta}{\fontsize{\figurefontsize}{12}$\Delta$}
\psfrag{Lambda}{\fontsize{\figurefontsize}{12}$\Lambda$}
\includegraphics[scale=\scalesize]{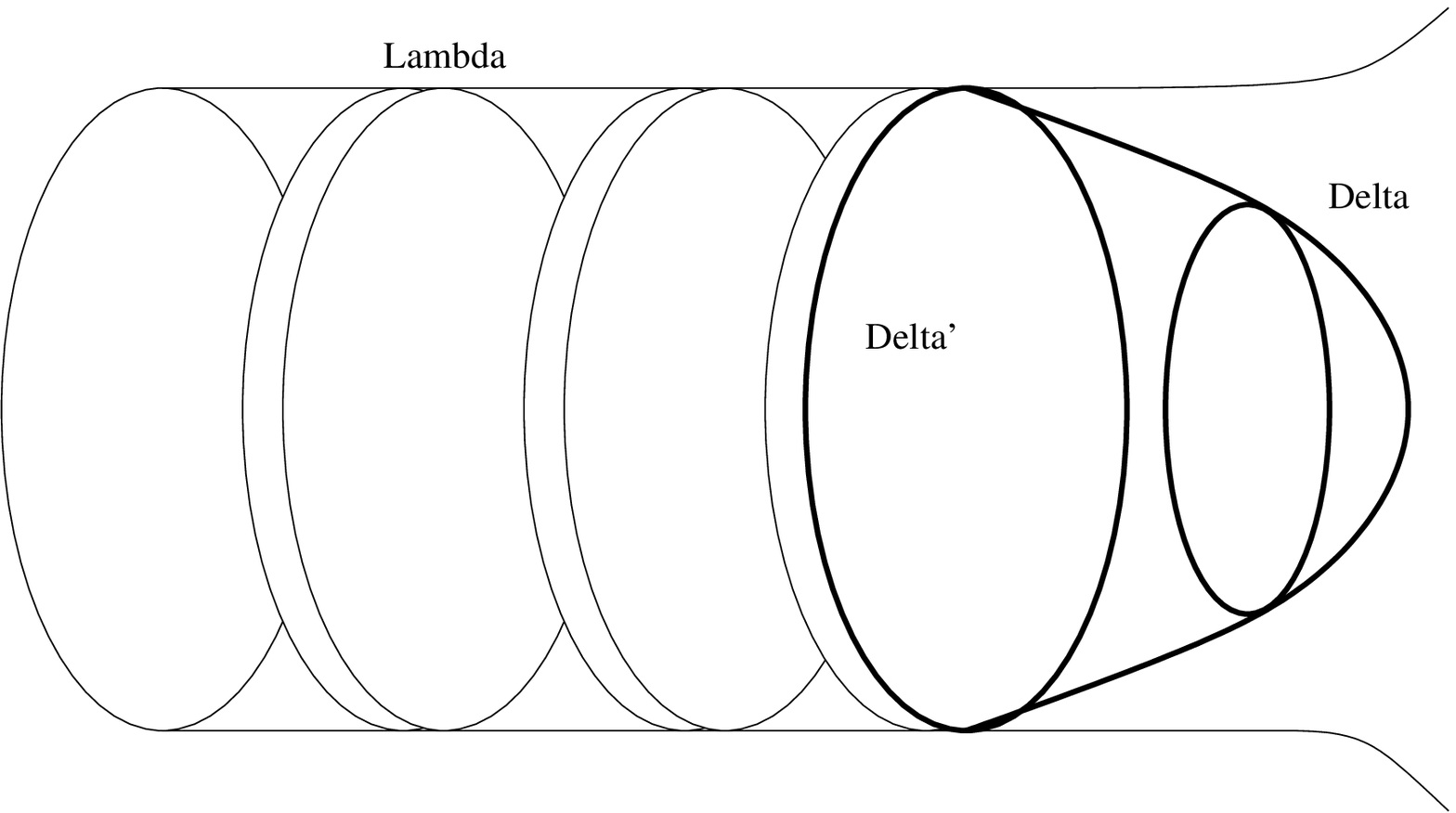}}
\quad \subfigure{\includegraphics[scale=\scalesize]{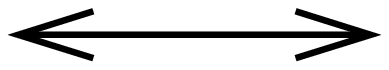}} \quad
\subfigure{\psfrag{Delta'}{\fontsize{\figurefontsize}{12}$\Delta'$}
\psfrag{D}{\fontsize{\figurefontsize}{12}$D$}
\psfrag{A'}{\fontsize{\figurefontsize}{12}$A'$}
\psfrag{Lambda}{\fontsize{\figurefontsize}{12}$\Lambda$}
\includegraphics[scale=\scalesize]{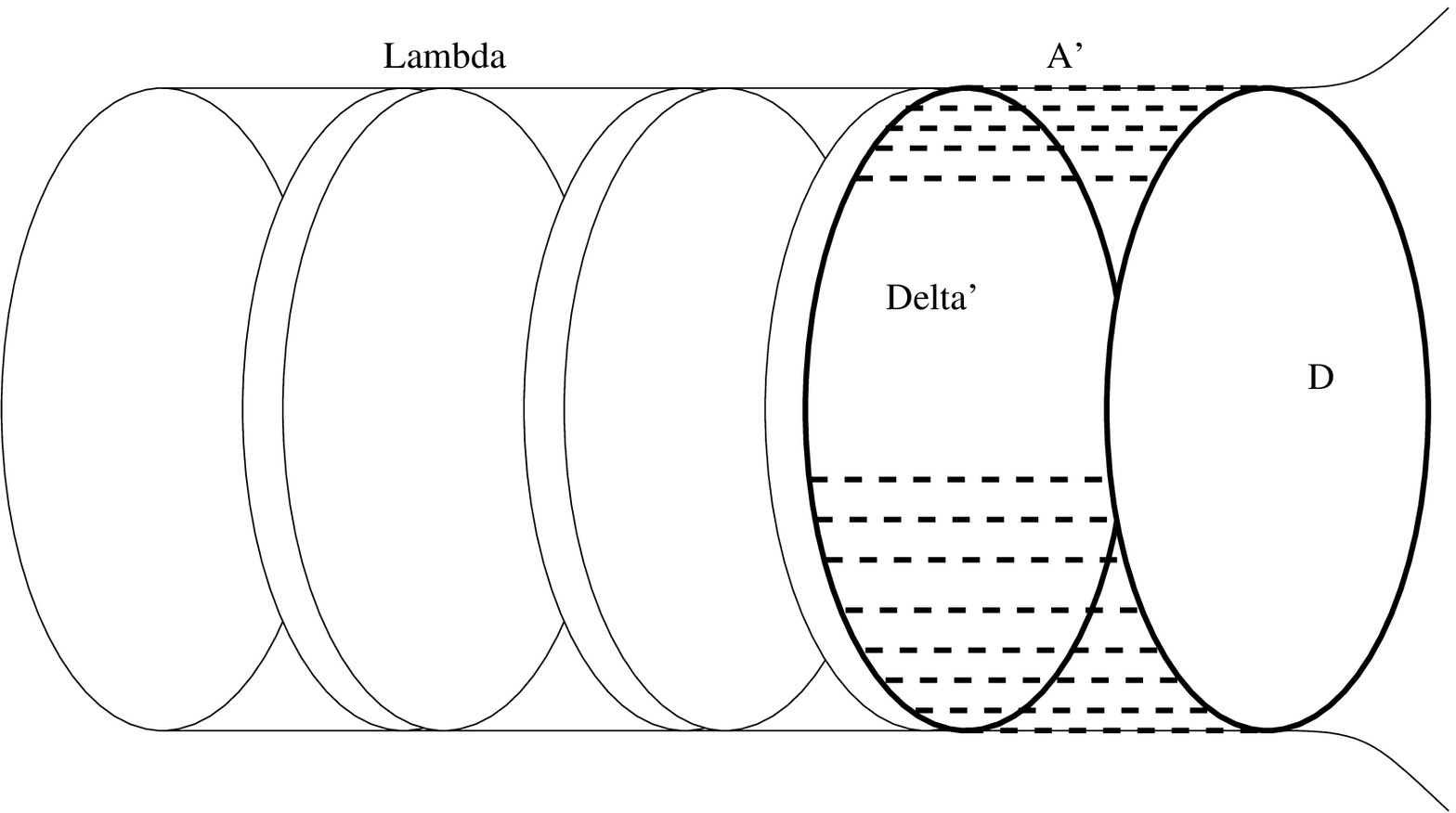}}}
\caption{\small A $1$--handle of $H_n$, being intersected by $\Lambda$. On the left it shows a strong tightening disc $\Delta$. On the right a parallel tightening disc $D$. The figure illustrates how to obtain one from the other.}\label{F:strong-parallel}
\end{figure}
}

Now suppose $D'$ is an $i$--tightening disc. Then $D=f^{-i}(D')$ is a $0$--tightening disc. Let $n$ be such that
$(D,\partial D)\subseteq(H_n,\partial H_n)$. Therefore $D$ is parallel to some $E\in\mathcal{E}_n$. But $D$ is disjoint
from $\mathcal{E}_n$, therefore there is an annulus in $A\subseteq\partial H_n$ from $\partial D$ to $\partial E$. In
fact consider $A\cap\Lambda$, which consists of concentric essential closed curves (because $\Lambda$ is carried by
$\mathcal{E}_n$). Let $A'\subseteq A$ be the ``smallest'' annulus connecting $\partial D$ to $\Lambda\cap\partial H_n$. Now consider $D''=A'\cup D$ and push it (rel $\partial D''$) slightly into $H_n$, obtaining $\Delta$. Such a disc is not yet a strong tightening disc because $\Delta\cap H_j$, $2\leq j\leq n-1$ need not consist of incidence discs. But that can be arranged by applying  Lemma \ref{L:make_incidence_tool} repeatedly on successive $H_j$'s with $j$ decreasing from
$n-1$ to $2$ (see Remark
\ref{rm:make_incidence_improvement}).
\end{obs}

\begin{corollary}\label{C:parallel_equivalence}
There exists a tightening disc if and only if there exists a parallel tightening disc with respect to some $H_i$.
\end{corollary}
\begin{proof}
If there exists a tightening disc use Proposition \ref{P:equivalence} and obtain a strong tightening disc. The
construction described on Remark \ref{rm:parallel} then yields a $0$--tightening disc. Given an $i$--tightening disc the construction of Remark \ref{rm:parallel} also yields a tightening disc.
\end{proof}

Supported by the corollary we shall henceforth deal only with parallel tightening discs. In fact, we will deal only
with $i$--tightening discs. Thus for the remainder of argument the reader can forget the
definitions of ``tightening disc'' and ``strong tightening disc''.

\begin{obs}\label{rm:f-preserves-tightness}
Every $(i+1)$--tightening disc is also an $i$--tightening disc. Therefore it is clear that if $\Delta$ is an $i$--tightening disc then so is $f(\Delta)$.
\end{obs}

\section{Push-away}\label{S:pushaway}

Recall from the introduction (Section \ref{SS:intro:strategy}) that given an $i$--tightening disc $\Delta$ it would
be desirable to obtain a framed isotopy $\phi$ such that $(\phi\circ f)^p(\Delta)\cap\Delta=\emptyset$. The following describes a step toward achieving this property.

\begin{definition}\label{D:pushaway}
Let $\Delta$ be a disc and $S$ a disjoint union of discs, both
embedded in a compact $3$--manifold $M$. Assume that
$\partial\Delta\cap\partial S=\emptyset$. Consider
$\gamma\subseteq\Delta\cap S$ a component which is innermost in
$\Delta$, bounding a disc $D\subseteq\Delta$. Since $S$ is
incompressible $\gamma$ also bounds a disc $D'\subseteq S$. Let
$S'$ be the surface obtained from $S$ by replacing $D$ with $D'$
(and pushing it a bit away from $\Delta$). We say that {\em $S'$
is obtained from $S$ by surgery along $\Delta$ and on $\gamma$}.
If $S''$ is obtained from $S$ by a sequence of surgeries along
$\Delta$ and $S''\cap\Delta=\emptyset$ then it is obtained by {\em
pushing $S$ away from $\Delta$}, see Figure \ref{F:push-away}. We call such a move a {\em
push-away}.
\end{definition}

{\def\scalesize{0.35}

\begin{figure}[ht]
\centering
\mbox{\subfigure{\psfrag{Delta}{\fontsize{\figurefontsize}{12}$\Delta$}
\psfrag{S}{\fontsize{\figurefontsize}{12}$S$}
\includegraphics[scale=\scalesize]{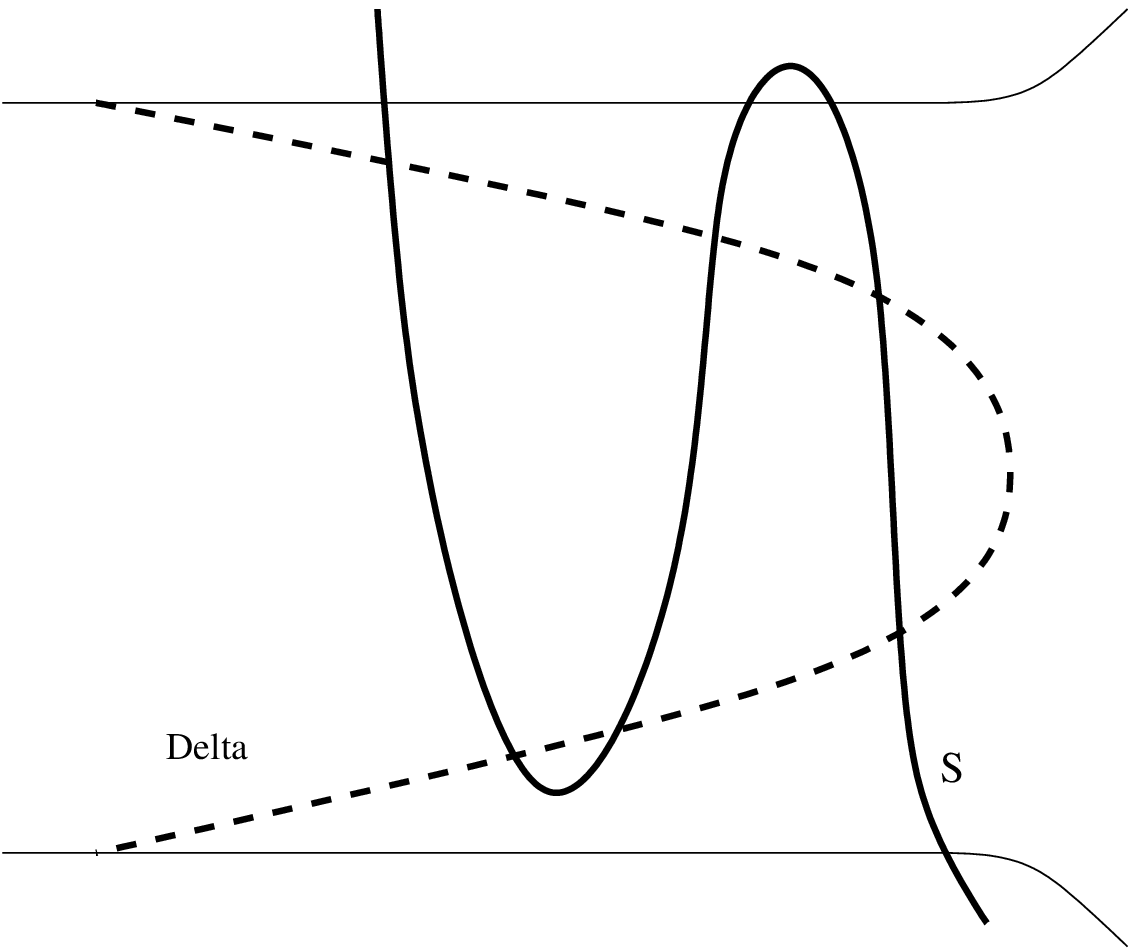}}
\quad \subfigure{\includegraphics[scale=\scalesize]{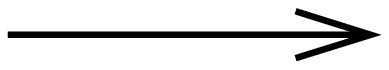}} \quad
\subfigure{\psfrag{Delta}{\fontsize{\figurefontsize}{12}$\Delta$}
\psfrag{S''}{\fontsize{\figurefontsize}{12}$S''$}
\includegraphics[scale=\scalesize]{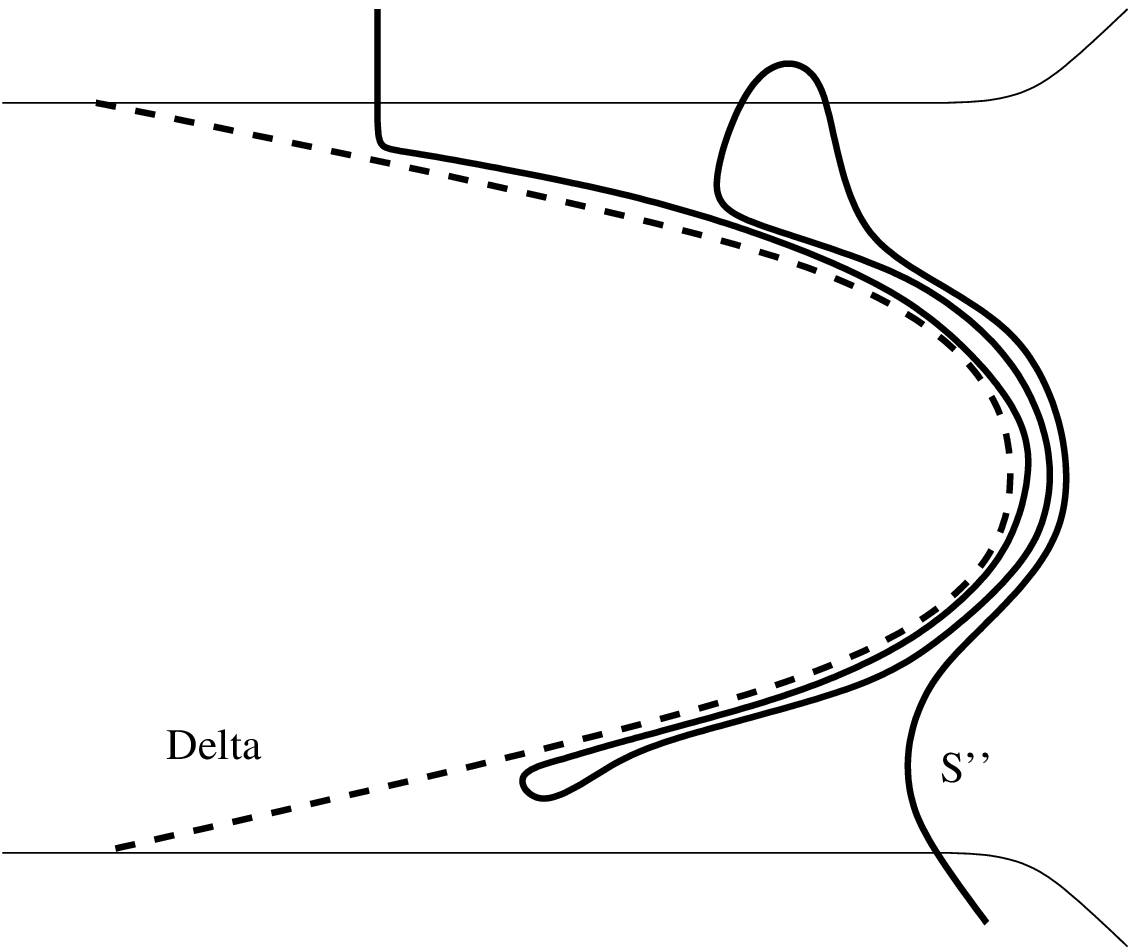}}}
\caption{\small A push-away from the disc $\Delta$, represented as the dashed line. The push-away of $S$ from $\Delta$ yields $S''$ disjoint from $\Delta$.}\label{F:push-away}
\end{figure}
}

\begin{obs}\label{R:push-away1}
In an irreducible manifold $M$ a push-away is realizable by an isotopy. In the present context of an irreducible
automorphism $f\colon H\to H$ if $\Delta\subseteq H_0$ is disjoint from $\mathcal{E}$ then a push-away is realizable by a framed
isotopy.
Also note that a push-away is a solely geometric operation, disregarding the measure $(\Omega,\nu)$ or any tightness
property.
\end{obs}

\begin{proposition}\label{P:well-defined}
Given $\Delta$, $S$, as in the Definition \ref{D:pushaway} the surface $S''$ obtained from $S$ by pushing it away from
$\Delta$
is uniquely determined (up to a small isotopy in a neighborhood of $\Delta$).
\end{proposition}

\begin{proof}
We have to show that $S''$ does not depend on the choice of sequence of surgeries on components of intersection of $S$ with $\Delta$.
We prove this by induction on $|\Delta\cap S|$.

If $|\Delta\cap S|=1$ there is just one surgery performed and the result holds.
Therefore now assume it holds for any pair of discs with fewer than $n$ intersection components.

Suppose then that $|\Delta\cap S|=n$. Let $S_0$, $S_1$ be two surfaces obtained from $S$ by a single surgery along
$\Delta$,
as described in Definition \ref{D:pushaway}. These surgeries are performed on curves $\gamma_0$,
$\gamma_1\subseteq\Delta\cap S$
which are innermost in $\Delta$, $\gamma_0$ yielding $S_0$ and $\gamma_1$ yielding $S_1$.
Since surgery along $\Delta$ does not introduce intersection components we only need to verify that by pushing $S_0$ and
$S_1$
away from $\Delta$ yields the same surface $S''$. There are two cases to consider:

\begin{itemize}
    \item [-] $\gamma_0$, $\gamma_1$ are ``disjoint'' in $S$, in the sense that the corresponding discs $D_0$,
$D_1\subseteq S$
they bound are disjoint. In this case $\gamma_0\subseteq D_0\subseteq S_1$, $\gamma_1\subseteq D_1\subseteq S_0$.
Now let $T_0$ be obtained by surgery along $\Delta$ on $\gamma_1$. Similarly, $T_1$ is obtained by surgery on
$\gamma_0$.
It is clear that $T_0=T_1$. The induction hypothesis applies to $T_0$, $T_1$,
therefore pushing them away from $\Delta$ yields a single $S''$. But the hypothesis also applies to $S_0$, $S_1$,
so pushing away also yields the same $S''$, proving the result does not depend on the choice between
$\gamma_0$, $\gamma_1$ for the first surgery.
    \item [-] $\gamma_0$, $\gamma_1$ are ``concentric'' in $S$, in the sense that the corresponding discs
$D_0$, $D_1\subseteq S$ they bound are not disjoint.
In this case one is (properly) contained in the other.
Assume without loss of generality that $D_1\subseteq D_0$. In this case $\gamma_1\subseteq S_0$ but $\gamma_0$,
$\gamma_1$
are disjoint from $S_1$. Let $T_0$ be obtained from $S_0$ by surgery on $\gamma_1$ along $\Delta$. Then $T_0=S_1$ and,
as in the previous case, they yield the same $S''$, completing the proof.
\end{itemize}
\end{proof}

\begin{obs}\label{R:push-away2}
It is clear that if $C$ is a component of $S$
then $C''$ is obtained by pushing $C$ away from $\Delta$ if  and
only if $C''$ is a component of $S''$.
\end{obs}

The following is more technical.

\begin{lemma}\label{L:inherited_push-away}
Suppose $(\Delta,\partial\Delta)\subseteq(H_0,\partial H_0)$ and $(S,\partial S)\subseteq (H_1,\partial H_1)$ are
$i$--transverse  incidence discs for some $i\in\mathbb{Z}$. Also assume $\partial\Delta\cap S=\emptyset$. Let
$D\subseteq S\cap H_0$ be a component. If $D''$, $S''$ are obtained from $D$, $S$ by push-away from $\Delta$ then
$D''\subseteq S''$ and both are $i$--transverse. In particular $S''\cap H_0$ is obtained by pushing $S\cap H_0$ away
from $\Delta$.
\end{lemma}

\begin{proof}
Let $\gamma\subseteq \Delta\cap S$. From $\Delta\subseteq H_0$ it follows that $\gamma\subseteq H_0$.
Therefore $\gamma$ is contained in a component of $S\cap H_0$, which is an essential disc $D\subseteq S\cap H_0$ because
$S$ is an incidence disc. This means that the disc $F'\subseteq S$ bounded by $\gamma$ is also contained in $D$. Hence if
$D'$, $S'$ are obtained from $D$, $S$
by surgery on $\gamma$ along $\Delta$ then $D'\subseteq S'\cap H_0$, proving the result for $D''$, $S''$ pushed away
from $\Delta$. Since both $D$ and $\Delta$ are $i$--transverse $\Delta\cap D$ does not intersect $H_i$, hence
$\gamma\cap H_i=\emptyset$. Therefore both discs $F$, $F'$ bounded by $\gamma$ are also $i$--transverse and the surgery yields an
$i$--transverse disc.
\end{proof}

\section{The proof}

\subsection{Main construction}\label{S:proof:main}

Assume that the lamination is not tight. Now consider
$i$--tightening discs $\Delta$ with the properties that 1)
$(\Delta,\partial \Delta)\subseteq (H_0,\partial H_0)$, and 2)
$\Delta$ has minimal ``height'' $h_0$. The {\em height} of a disc $(\Delta,\partial\Delta)\subseteq(H_0,\partial H_0)$ is the greatest $h_0=-i_0$ such that $\Delta$ is transverse to $\mathcal{H}_{i_0}$. Put another way, $\Delta\cap H_{i_0}$
consists of discs transverse to $\mathcal{H}_{i_0}$, but for any $i>i_0$
the intersection $\Delta\cap H_i$ contains a component which is not transverse of $\mathcal{H}_i$.

Now among $i$-tightening discs of minimal height $h_0=-i_0$ described above, choose one, also labeled $\Delta$, having minimal intersection
$\inter{\Delta}{\mathcal{H}_{i_0}}{\hat\nu^{i_0}}$. Note that such a disc exists since $\inter{\Delta}{\mathcal{H}_{i_0}}{\hat\nu^{i_0}}$ is an integer and non-negative combination of entries of the vector $\hat\nu^{i_0}=\lambda^{i_0}\hat\nu>0$ (see equation (\ref{eq:intersection-n}) in Subsection \ref{SS:intro:terminology} above). We fix this $\Delta$ throughout the remainder of the paper.

Recall the strategy outlined in the introduction, Subsection \ref{SS:intro:strategy}. The goal is to find a suitable
enlargement $\mathcal{E}'$ of the disc system $\mathcal{E}\subseteq\mathcal{E}'\subseteq H_0$ including $\Delta\in\mathcal{E}'$ and an isotopy $\phi$
so that $\mathcal{E}'$ is admissible with respect to $\phi\circ f$. Such an $\mathcal{E}'$ (and $\phi$) is obtained by
considering the union of the $f^p(\Delta)\cap H_0$ with increasing $p\geq 0$. In general these will be neither discs nor
pairwise disjoint. By using mainly push-aways and standard incidence moves we obtain an ambient isotopy $h$ such that the
union of these $(h\circ f)^p(\Delta)\cap H_0$ indeed consists of an nested increasing sequence of disc systems, all
transverse to $\mathcal{H}_{i_0}$ (Proposition \ref{P:main_construction}). This process eventually ``stabilizes'', in a sense which will be made precise (see Lemma \ref{L:stable}) and can be made finite, yielding the disc system $\mathcal{E}'$ and ambient isotopy $\phi$, see
Proposition \ref{P:stabilized} (in the statement the role of $\mathcal{E}'$ is played by $\mathcal{E}\cup\mathcal{D}$).

It is important to ensure that in the process $f$ does not become less efficient. In other words, the introduction of the the isotopy $\phi$ should not
yield $\phi\circ f$ less efficient than $f$. As mentioned before $\phi$ is essentially a composition of standard incidence
moves and push-aways. While studying standard incidence moves we proved many results assuring they do not increase
intersection of certain surfaces with $(\mathcal{H}_{i_0},\hat\nu^{i_0})$. We need something similar for push-aways. Below
$\Delta$ is the $i_0$--tightening disc fixed at the beginning of the section. Its minimality properties are important.

\begin{lemma}\label{L:push-away_reduce}
Suppose that $D$ is an embedded disc transverse to $\mathcal{H}_{i_0}$ with $D\cap\partial\Delta=\emptyset$. Let $D'$
be obtained by pushing $D$ away from $\Delta$. Then $D'$ is transverse to $\mathcal{H}_{i_0}$ and,
\[
\inter{D'}{\mathcal{H}_{i_0}}{\hat\nu^{i_0}}\leq \inter{D}{\mathcal{H}_{i_0}}{\hat\nu^{i_0}}.
\]
\end{lemma}
\begin{proof}
It suffices to show that a surgery along $\Delta$ on a component $\gamma\subseteq D\cap\Delta$ does not increase
intersection.
Let $F\subseteq D$, $F'\subseteq\Delta$ be the discs bounded by $\gamma$ in each disc (where $F'$ is innermost).
If $\inter{F'}{\mathcal{H}_{i_0}}{\hat\nu^{i_0}}\leq\inter{F}{\mathcal{H}_{i_0}}{\hat\nu^{i_0}}$ then the surgery --- replacing $F$
with $F'$ on $D$ --- does not increase intersection. Therefore assume otherwise that
$\inter{F'}{\mathcal{H}_{i_0}}{\hat\nu^{i_0}}>\inter{F}{\mathcal{H}_{i_0}}{\hat\nu^{i_0}}$. We shall obtain another
$i_0$--tightening disc
$\Delta'\subseteq H_0$, hence with height $h_0=-i_0$, and with $\inter{\Delta'}{\mathcal{H}_{i_0}}{\hat\nu^{i_0}}<\inter{\Delta}{\mathcal{H}_{i_0}}{\hat\nu^{i_0}}$, contradicting
the minimal intersection choice of $\Delta$.

Let $S$ be obtained by removing $F'$ from $\Delta$ and gluing in $F$. Then
$\inter{S}{\mathcal{H}_{i_0}}{\hat\nu^{i_0}}<\inter{\Delta}{\mathcal{H}_{i_0}}{\hat\nu^{i_0}}$. It may not be embedded, in case
$\gamma$ is not innermost in $D$. But in this case the singular set consists only of double curves (which, in this case, are simple closed curves). Then
perform finitely many double curve surgeries (as in the proof of the loop theorem), each changing $S$ in a neighborhood
of $\gamma$, possibly discarding a component without boundary. Recall that both $D$ and $\Delta$ are transverse to
$\mathcal{H}_{i_0}$, therefore so are $F$, $F'$, and hence $\gamma\cap H_{i_0}=\emptyset$. This shows that changing $S$
in neighborhood of $\gamma$ does not affect intersection with $(\mathcal{H}_{i_0},\hat\nu)$ (though discarding closed
surface components may reduce intersection). After finitely many such surgeries we obtain an embedded disc $S'$ with
$\inter{S'}{\mathcal{H}_{i_0}}{\hat\nu^{i_0}}<\inter{\Delta}{\mathcal{H}_{i_0}}{\hat\nu^{i_0}}$ and $\partial S'=\partial\Delta$. By
Lemma \ref{L:make_incidence} a sequence of standard incidence moves yield an $i_0$--transverse incidence disc $\Delta'$
with $\partial\Delta'=\partial\Delta$ and:
\[
\inter{\Delta'}{\mathcal{H}_{i_0}}{\hat\nu^{i_0}}<\inter{\Delta}{\mathcal{H}_{i_0}}{\hat\nu^{i_0}}.
\]
Therefore $\Delta'\subseteq H_0$ is also an $i_0$--tightening disc, hence with height $h_0=-i_0$, contradicting the minimality hypothesis on $\inter{\Delta}{\mathcal{H}_{i_0}}{\hat\nu^{i_0}}$.
\end{proof}

For considering enlargements of disc systems we introduce the following notion. We say that two disc systems $\mathcal{D}$, $\mathcal{D}'\subseteq H_n$ are {\em compatible} if each disc $D\in\mathcal{D}$ is either a disc of $\mathcal{D}'$ or disjoint from $\mathcal{D}'$ (note that the definition is symmetric). It is easy to see that $\mathcal{D}\cup\mathcal{D}'$ is a disc system if, and only if, the systems are compatible.

The next proposition describes the main construction of the section, taming the enlargement of the system resulting from adding in the tightening disc $\Delta$ and its iterates under powers of the automorphism.

\begin{proposition}\label{P:main_construction}
There exists a nested sequence of systems of $i_0$--transverse incidence discs
\[
\{\Delta\}=\mathcal{D}_0\subseteq\mathcal{D}_1\subseteq\cdots\subseteq H_0,
\]
and framed-isotopies $h_i\colon H\to H$, $i\geq 0$, in $H_0$  such that:
\begin{itemize}
    \item $\mathcal{D}_{i+1}=\mathcal{D}_0\cup (h_{i}\circ f(\mathcal{D}_{i})\cap H_0)$;
    \item $\mathcal{D}_i$ is disjoint from $\mathcal{E}$;
    \item if $D\in\mathcal{D}_i$ then $\inter{h_i\circ f(D)}{\mathcal{H}_{i_0}}{\hat\nu^{i_0}}\leq
\lambda\cdot\left(\inter{D}{\mathcal{H}_{i_0}}{\hat\nu^{i_0}}\right) $.
\end{itemize}
\end{proposition}
\begin{proof}
The construction will be recursive. We illustrate it by considering its first step. Let $\mathcal{D}_0=\{\Delta\}$ and consider $f(\Delta)\cap H_0$, which consists of $i_0$--transverse (in fact $(i_0+1)$--transverse) essential discs because $\Delta$ is an $i_0$--tightening disc. If
$f(\Delta)\cap\Delta=\emptyset$ set $\mathcal{D}_1=\mathcal{D}_0\cup (f(\Delta)\cap H_0)$ and $h_0=\Id$.

If $f(\Delta)\cap\Delta\neq\emptyset$ a push-away is necessary. We claim that $f(\Delta)\cap\Delta$ can be assumed to consist of closed curves. Indeed, recall that $\Delta$ is parallel to a disc $E\in\mathcal{E}$. Therefore
$\partial\Delta$ and $\partial E$ co-bound an annulus in $\partial H_0$. By sliding $\partial\Delta$ along this annulus we can assume $\partial\Delta$ is sufficiently close to $\partial E$ to to ensure that it avoids $f(\Delta)$ (for purposes of understanding push-aways it may be helpful to regard $\Delta$ with $\partial\Delta\subseteq E$, see Remark \ref{rm:parallel}).
Hence $f(\Delta)\cap\partial\Delta=\emptyset$ and $f(\Delta)\cap H_0$ can be pushed away from $\Delta$, yielding
$i_0$--transverse discs (Lemma \ref{L:inherited_push-away}). It is realizable by a framed isotopy $h$ such that $h\circ f(\Delta)\cap\Delta=\emptyset$, see Remark \ref{R:push-away1}. Moreover $h$ does not increase intersection of $f(\Delta)$ with $(\mathcal{H}_{i_0},\hat\nu)$. Now while $h\circ f(\Delta)\cap H_0$ consists of essential discs these need not be incidence discs. By Lemma \ref{L:make_incidence} we obtain a framed isotopy $g$ such that $g(\Delta)=\Delta$ and $g\circ h\circ
f(\Delta)\cap H_0$ consisting of incidence discs disjoint from $\mathcal{D}_0=\{\Delta\}$. Let $\mathcal{D}_1 = \mathcal{D}_0\cup(h_0 \circ f(\Delta)\cap H_0)$, where $h_0=g\circ h$. Moreover applying $g\circ h$ does not increase intersection of $f(\Delta)$ with $(\mathcal{H}_{i_0},\hat\nu)$. We thus obtained $\mathcal{D}_1$ and $h_0$ as in the statement.

To generalize this construction so it can be applied recursively we enlarge the list of properties of $\mathcal{D}_i$ given in the statement. Suppose that there exists $\mathcal{D}_i$ and $h_i$, $0\leq i\leq j$ satisfying:
\begin{enumerate}
    \item $\mathcal{D}_i\cap\mathcal{E}=\emptyset$;
    \item $(h_i\circ f(\mathcal{D}_i))\cap H_0$ consists of $i_0$--transverse incidence discs compatible with
$\mathcal{D}_i$;
    \item if $\{ D\}$ is compatible with $\mathcal{D}_i$ then $\inter{h_i\circ f(D)}{\mathcal{H}_{i_0}}{\hat\nu^{i_0}}\leq
\lambda\cdot\left(\inter{D}{\mathcal{H}_{i_0}}{\hat\nu^{i_0}}\right)$.
\end{enumerate}
Suppose moreover that for any $0\leq i\leq j-1$ the following statements also hold:
\begin{itemize}
    \item[(A$_1$)] $\mathcal{D}_{i+1}=\mathcal{D}_0\cup\left( h_{i}\circ f(\mathcal{D}_{i})\cap H_0\right)$;
    \item[(A$_2$)] $\mathcal{D}_{i+1}=\mathcal{D}_i\cup\mathcal{D}'$, where $\mathcal{D}'=\mathcal{D}_i'$ is disjoint from (hence compatible with) $\mathcal{D}_i$. In particular $\mathcal{D}_{i}\subseteq\mathcal{D}_{i+1}$;
    \item[(B$_1$)] $h_{i+1}=\psi_i\circ h_i$ where $\psi_i$ is a framed isotopy such that,
    \item[(B$_2$)] $\psi_i$ leaves $h_i\circ f(\mathcal{D}_i)$ fixed: $\psi_i\circ h_i\circ f(\mathcal{D}_i)=h_i\circ f(\mathcal{D}_i)$.
\end{itemize}

By setting $i=j$ the equation in (A$_1$) defines a $\mathcal{D}_{j+1}$ for which we will verify the other properties. We shall also need a candidate for $\psi_j$ which defines $h_{j+1}$ by (B$_1$).  First, we prove (A$_2$) for $\mathcal{D}_{j+1}$, defined by (A$_1$):
\[
\mathcal{D}_{j+1}\stackrel{\text{def}}{=}\mathcal{D}_0\cup\left( h_{j}\circ f(\mathcal{D}_{j})\cap H_0\right)\stackrel{\text{(A$_2$)}}{=}
\mathcal{D}_0\cup\left( h_{j}\circ f(\mathcal{D}_{j-1}\cup\mathcal{D}')\cap H_0\right).
\]
But by (B$_1$), $h_{j}=\psi_{j-1}\circ h_{j-1}$:
\begin{multline*}
\mathcal{D}_{j+1}=\mathcal{D}_0\cup\big( h_{j}\circ f(\mathcal{D}_{j-1}\cup\mathcal{D}')\cap H_0\big)\stackrel{\text{(B$_1$)}}{=}
\mathcal{D}_0\cup\big((\psi_{j-1}\circ h_{j-1}\circ f)(\mathcal{D}_{j-1}\cup\mathcal{D}')\cap H_0\big)=\\
\mathcal{D}_0\cup\big((\psi_{j-1}\circ h_{j-1}\circ f)(\mathcal{D}_{j-1})\cap H_0\big)\cup\big((\psi_{j-1}\circ h_{j-1}\circ f)(\mathcal{D}')\cap H_0\big)\stackrel{\text{(B$_2$)}}{=}\\
\mathcal{D}_0\cup\big((h_{j-1}\circ f)(\mathcal{D}_{j-1})\cap H_0\big)\cup\big((\psi_{j-1}\circ h_{j-1}\circ f)(\mathcal{D}')\cap H_0\big)\stackrel{\text{(A$_1$)}}{=}\\
\mathcal{D}_j\cup\big((\psi_{j-1}\circ h_{j-1}\circ f)(\mathcal{D}')\cap H_0\big).
\end{multline*}

Setting $\mathcal{D}''=(\psi_{j-1}\circ h_{j-1}\circ f)(\mathcal{D}')\cap H_0$ yields
\[
\mathcal{D}_{j+1}=\mathcal{D}_j\cup\mathcal{D}'',
\]
where $\mathcal{D}''=\mathcal{D}_j''$ is disjoint from $\mathcal{D}_j$, proving (A$_2$).

Consider then the decomposition
$\mathcal{D}_{j+1}=\mathcal{D}_j\cup\mathcal{D}''$, hence $h_j\circ f(\mathcal{D}_{j+1})=h_j\circ
f(\mathcal{D}_j\cup\mathcal{D}'')=h_j\circ f(\mathcal{D}_j)\cup h_j\circ f(\mathcal{D}'')$. In particular $h_j \circ f(\mathcal{D}'')$ is compatible with $h_j \circ f(\mathcal{D}_j)$. By property (3), if
$D\in\mathcal{D}''$ then $h_j\circ f(D)$ does not increase intersection with $({\mathcal{H}_{i_0}},{\hat\nu^{i_0}})$.

By properties (A$_1$) and (A$_2$) we see that if $h_j\circ f(\mathcal{D}_{j+1})\cap H_0$ is not compatible with
$\mathcal{D}_{j+1}$ then $h_j\circ f(\mathcal{D}'')\cap\mathcal{D}_0\neq\emptyset$, i.e. $h_j\circ
f(\mathcal{D}'')\cap\Delta\neq\emptyset$. In this case push $h_j\circ f(\mathcal{D}'')$ away from $\Delta$. Using Lemma
\ref{L:push-away_reduce} this is realizable by a framed isotopy $h$, so $h\circ h_j\circ
f(\mathcal{D}')\cap\Delta=\emptyset$. Also it does not increase intersection with $(\mathcal{H}_{i_0},\hat\nu)$ and
leaves $h_j\circ f(\mathcal{D}_j)$ fixed. Now $h\circ h_j\circ f(\mathcal{D}')$ consists of a system disjoint from
$\mathcal{D}_j$. Applying
Lemma \ref{L:make_incidence_tool} obtain a framed isotopy $g$ such that  $g\circ h\circ h_j\circ f(\mathcal{D}')$
consists of incidence discs and $g$ fixes $\mathcal{D}_{j+1}$. Therefore $\psi_j=g\circ h$ fixes $h_j\circ
f(\mathcal{D}_j)\subseteq\mathcal{D}_{j+1}$.
Let $h_{j+1}=\psi_j\circ h_j$, hence satisfying (B$_1$) and (B$_2$). The construction yields properties (1)--(3) for $i=j+1$.
\end{proof}

The proposition yields the nested sequence $\{\Delta\}=\mathcal{D}_0\subseteq\mathcal{D}_1\subseteq\cdots\subseteq H_0$ of systems compatible with $\mathcal{E}$. Fix $j\geq 0$ and consider $\mathcal{D}_j \subseteq H_0$, which consists of a system of
$i_0$-transverse incidence discs.
Since there is an upper bound on the number of
isotopy classes of such discs there exists $j=J\in\mathbb{Z}$ which
``stabilizes'' these classes of discs in the following sense: $J$ has
the property that each disc of $\mathcal{D}_{J+1}\subseteq H_0$ is parallel to
one of $\mathcal{D}_J$. We shall fix such a $J$, but another ``stabilization'' property on $J$, which will be
related to intersection with $(\Omega,\nu)$, is required.

To say that $J$ {\em stabilizes} $\mathcal{D}_J\subseteq H_0$ means the following. If $D'$ is a
component of $\mathcal{D}_{J+1}\subseteq H_0$ then either $D'$ is transverse to $\mathcal{H}_0$ or there is a component $D$ of $\mathcal{D}_J$ such that:
\begin{enumerate}
    \item $D'$ is parallel to $D$, and
    \item $\inter{D}{\mathcal{H}_{i_0}}{\hat\nu^{i_0}}\leq\inter{D'}{\mathcal{H}_{i_0}}{\hat\nu^{i_0}}$.
\end{enumerate}

\begin{lemma}\label{L:stable}
There exists $J$ stabilizing $\mathcal{D}_J$.
\end{lemma}
\begin{proof}

As already mentioned the first property can be realized. Moreover, for any $j\geq 0$,
the system $\mathcal{D}_j$ consists of $i_0$--transverse discs. But the intersection of such a disc $D$ with
$(\mathcal{H}_{i_0},\hat\nu)$ is an integer and non-negative combination of the entries of $\hat\nu$, which are also all non-negative. Hence the set
$P_D=\{ \inter{D'}{\mathcal{H}_{i_0}}{\hat\nu^{i_0}}\,|\, D'$ isotopic to $D$ and $i_0$--transverse$\}$ is well-ordered. Therefore there exists a stabilizing $J$.
\end{proof}

\begin{obs}\label{obs:framed-isotopy}
If $\mathcal{D}_J\subseteq H_0$ is stabilized then there exists framed isotopy $\psi\colon H\to H$ such that $\psi(\mathcal{D}_{J+1})$ is carried by $\mathcal{E}\cup\mathcal{D}_J\subseteq H_0$, realizing the parallelism. Note $\psi$ thus has support in $H_1$, not necessarily in $H_0$.
\end{obs}
We have therefore:

\begin{proposition}\label{P:stabilized}
There exists system of $i_0$--transverse incidence discs $\mathcal{D}\subseteq H_0$ and framed isotopy $\phi$ in $H_1$ satisfying the following:
\begin{itemize}
    \item ${\mathcal{D}}\cap\mathcal{E}=\emptyset$;
    \item $\Delta\in\mathcal{D}$ is parallel to $E=E_1\in\mathcal{E}$;
    \item $\inter{\Delta}{\mathcal{H}_{i_0}}{\hat\nu^{i_0}}<\inter{E_1}{\mathcal{H}_{i_0}}{\hat\nu^{i_0}}$;
    \item if $D\in\mathcal{D}$ then $\phi\circ f(D)\cap H_0$ is carried by $\mathcal{E}\cup\mathcal{D}$ and
        \[
        \inter{\phi\circ f(D)}{\mathcal{H}_{i_0}}{\hat\nu^{i_0}}\leq
\inter{f(D)}{\mathcal{H}_{i_0}}{\hat\nu^{i_0}}=\lambda\cdot(\inter{D}{\mathcal{H}_{i_0}}{\hat\nu^{i_0}}).
        \]
\end{itemize}
In particular $\mathcal{E}'=\mathcal{E}\cup\mathcal{D}$ is an admissible disc system (with respect to $\phi\circ f$) containing the tightening disc $\Delta$.
\end{proposition}
\begin{proof}
Make $\mathcal{D}=\mathcal{D}_J\subseteq H_0$, where $\mathcal{D}_J$ is stabilized. Also assume without loss of
generality that the
disc $E\in\mathcal{E}=\mathcal{E}_0=\{E_1,\dots, E_k\}$ parallel to $\Delta$ is $E=E_1$. Use Proposition
\ref{P:main_construction} letting $\phi$ be $h_J$ composed with the parallelism that takes $\mathcal{D}_{J+1}$ to be carried by $\mathcal{E}\cup\mathcal{D}=\mathcal{E}\cup\mathcal{D}_J$. Then $\phi$ is a framed isotopy in $H_1$, see Remark \ref{obs:framed-isotopy}.
\end{proof}

\begin{obs}\label{obs:stabilized}
Note that, since $\phi$ is a framed isotopy in $H_1$ then $f(\mathcal{E}_0)=\mathcal{E}_1$ and hence $\phi\circ f(\mathcal{E}_0)=f(\mathcal{E}_0)$.
\end{obs}

\subsection{Incidence matrices}\label{S:proof:matrices}

From the construction realized so far we can already obtain a representative of $[f]$ and corresponding disc system
which is more efficient than the original $f$ and $\mathcal{E}$. Verifying that it is indeed more efficient requires some further work.

Obtaining the new representative requires three steps:

1- Using Proposition \ref{P:stabilized} we have that the system $\mathcal{E}'=\mathcal{E}\cup\mathcal{D}$ is admissible
with respect to $\phi\circ f$. Also, $\mathcal{E}$ itself is admissible with respect to $\phi\circ f$.

2- Now let $\varphi$ realize the parallelism between $E_1$ and $\Delta$, in such a way that $\varphi(E_1)=\Delta$. Note
that $\varphi$
is not a framed isotopy. Consider $\hat f=\phi\circ f\circ \varphi$. The system $\mathcal{E}'$ is admissible with
respect to $\hat f$.

3- Pass to an irreducible sub-system $\mathcal{E}'$.

The task then is to verify that $\hat f$ is more efficient than $f$.

We consider the corresponding incidence matrices. Let $M=M(f,\mathcal{E})$ be the incidence matrix of the original pair
$(f,\mathcal{E})$:
\begin{equation}
 M=\left[
\begin{array}{c ccc }
e_{(1,1)} & & \cdots & e_{(1,k)}\\
\vdots & & \ddots  & \vdots \\
e_{(k,1)} & & \cdots & e_{(k,k)}
\end{array}\right]
\end{equation}
Recall that $e_{(i,j)}$ is the number of components of $f(E_i)\cap H_0$ carried by $E_j$.

Now let $\bar f=\phi\circ f$ and consider the incidence matrix $\bar M=\bar M(\bar f,\mathcal{E}\cup\mathcal{D})$:
\begin{equation}
 \bar M=\left[
\begin{array}{c c c | c c c}
e_{(1,1)} & \cdots & e_{(1,k)} & d_{(1,k+1)} & \dots & d_{(1,k+l)} \\
\vdots & \ddots  & \vdots & \vdots &\vdots & \vdots\\
e_{(k,1)} & \cdots & e_{(k,k)} & d_{(k,k+1)} & \dots & d_{(k,k+l)} \\ \cline{1-6}
d_{(k+1,1)} & \cdots & d_{(k+1,k)} & d_{(k+1,k+1)} & \dots & d_{(k+1,k+l)} \\
\vdots & \ddots  & \vdots & \vdots &\vdots & \vdots\\
d_{(k+l,1)} & \cdots & d_{(k+l,k)} & d_{(k+l,k+1)} & \dots & d_{(k+l,k+l)} \\
\end{array}\right]
\end{equation}
Here we are ordering $\mathcal{E}\cup\mathcal{D}=\{E_1,\dots,E_k,D_{k+1},\dots, D_{k+l}\}$, first those of $\mathcal{E}$ then those of $\mathcal{D}$.

Since $\phi$ is a framed isotopy $\bar f|_\mathcal{E}=f|_\mathcal{E}$, therefore the upper left block of $\bar M$ is
$M$.
Moreover $f(\mathcal{E})\cap H_0$ consists only of discs carried by $\mathcal{E}$, therefore there is a zero block in $\bar M$:
\begin{equation}
 \bar M=\left[
\begin{array}{c c c | c c c}
e_{(1,1)} & \cdots & e_{(1,k)} &  & &  \\
\vdots & \ddots  & \vdots &  & 0 &  \\
e_{(k,1)} & \cdots & e_{(k,k)} &  &  &  \\ \cline{1-6}
d_{(k+1,1)} & \cdots & d_{(k+1,k)} & d_{(k+1,k+1)} & \dots & d_{(k+1,k+l)} \\
\vdots & \ddots  & \vdots & \vdots &\vdots & \vdots\\
d_{(k+l,1)} & \cdots & d_{(k+l,k)} & d_{(k+l,k+1)} & \dots & d_{(k+l,k+l)} \\
\end{array}\right]
\end{equation}

The following is a corollary of Proposition \ref{P:stabilized}.
\begin{lemma}\label{L:early_subinvariance}
If $v\in \mathbb{R}^{k+l}$ is the vector whose entries are $v_i=\inter{E_i}{\mathcal{H}_{i_0}}{\hat\nu^{i_0}}>0$ ($1\leq i\leq k$),
$v_i=\inter{D_i}{\mathcal{H}_{i_0}}{\hat\nu^{i_0}}>0$ ($k+1\leq i\leq k+l$), then $\bar M
v\leq\lambda v$.
\end{lemma}
\begin{proof}
Proposition \ref{P:stabilized}.
\end{proof}

Note that in a sense $\bar f$ is still as efficient as $f$: indeed the inequality $\bar M v\leq\lambda v$ is really
non-strict. That is because $M$ is an irreducible sub-matrix whose growth rate is $\lambda$ (more precisely, $(\bar M
v)_i=\lambda v_i$ for $1\leq i\leq k$).

The real ``tightening'' move is represented by $\varphi$. Assume the last $D_{k+l}\in\mathcal{D}$ is the tightening disc
$\Delta$, parallel to the first the first disc $E_1\in\mathcal{E}$. Consider $\hat f=\bar f\circ \varphi=\phi\circ
f\circ\varphi$ and the corresponding matrix $\hat M$.
The isotopy $\varphi$ affects the incidence matrix $\bar M$ by copying the last line (that corresponding to $\Delta$)
over the first one (that corresponding to $E_1$).

\begin{equation}\label{eq:matrix_2}
 \hat M=\left[
\begin{array}{c c c | c c c}
d_{(k+l,1)} & \cdots & d_{(k+l,k)} & d_{(k+l,k+1)} & \dots & d_{(k+l,k+l)} \\ \cline{1-6}
e_{(2,1)} & \cdots & e_{(2,k)} &  & &  \\
\vdots & \ddots  & \vdots &  & 0 &  \\
e_{(k,1)} & \cdots & e_{(k,k)} &  &  &  \\ \cline{1-6}
d_{(k+1,1)} & \cdots & d_{(k+1,k)} & d_{(k+1,k+1)} & \dots & d_{(k+1,k+l)} \\
\vdots & \ddots  & \vdots & \vdots &\vdots & \vdots\\
d_{(k+l,1)} & \cdots & d_{(k+l,k)} & d_{(k+l,k+1)} & \dots & d_{(k+l,k+l)} \\
\end{array}\right]
\end{equation}
Recall the vector $v\in\mathbb{R}^{k+l}$ (see Lemma \ref{L:early_subinvariance}). It also holds that $\hat M v\leq \lambda v$ because $(\hat M v)_j=(\bar Mv)_j$
for $j>1$ and $(\hat M v)_1< (\bar M v)_1=\lambda v_1$, since $\Delta$ is tightening disc. To complete the construction it is
necessary to pass to an irreducible sub-system, say $\hat{\mathcal{D}}\subseteq\mathcal{E}\cup\mathcal{D}$. We can
easily extract from $v$ the vector $w$ corresponding to $\hat{\mathcal{D}}$ for which, letting $N$ be the corresponding
incidence matrix (with respect to $\hat f$), $N w\leq \lambda w$.  If there is $(N w)_j < \lambda w_j$ for some $j$ then
 $\lambda(\hat f)<\lambda$ by Proposition \ref{P:subinvariance}. The problem in finding such an $j$ is that the
sub-system $\hat{\mathcal{D}}$ may have little to do with $E_1$.

The solution then is to prove that the reduction obtained from $\varphi$, i.e. $(\hat M v)_1< (\bar M v)_1=\lambda v_1$,
on the first disc somehow ``propagates'' to all discs of $\mathcal{E}\cup\mathcal{D}$ through powers of $\hat f$.

\begin{claim}\label{Cl:claim}
For every $1\leq j\leq k+l$ there exists power $p\geq 1$ such that $(\hat M^p v)_j<\lambda^p v_j$.
\end{claim}

Therefore even after passing to the sub-system $\hat{\mathcal{D}}$ a reduction on the growth can be detected. Assuming
the claim we can prove Theorem \ref{Th:tightness}.

\begin{proof}[Proof of Theorem \ref{Th:tightness}]
We prove the contrapositive. If $f$ is not tight we obtain $\hat f=\phi\circ f\circ\varphi$ isotopic to $f$ and the matrix $\hat M$ constructed before. But the growth $\lambda(\hat f)$ of $\hat f$ is that $\lambda(N)$ of an irreducible submatrix $N$ of $\hat M$.
By Corollary \ref{C:submatrix} and Claim \ref{Cl:claim}, $\lambda(\hat f)=\lambda(N)<\lambda(f)$. Therefore $\hat f$ is more efficient than $f$, which is then not efficient.
\end{proof}

\subsection{Tracking the growth}\label{S:proof:tracking}
As seen so far it suffices to prove Claim \ref{Cl:claim} to obtain Theorem \ref{Th:tightness}.
We then go some steps back and obtain $\hat M$ in an alternative away, keeping track of its growth.

We shall build the following matrix $W$, which is an incidence matrix but not really of a disc system. Let
$\mathcal{S}^i=f^i(\mathcal{D})\cap H_0$, $0\leq i\leq h_0=-i_0$. If $i=0$, $1$ then $\mathcal{S}^i$ is a disc system,
because $\mathcal{D}$ consists of incidence discs. If $i=h_0$ that is also the case because $\mathcal{D}$ consists of
$i_0$--transverse discs (in fact we identify $\mathcal{S}^{h_0}$ with $\mathcal{E}$). But $\mathcal{S}^i$ need not
consist os essential discs\footnote{That could easily be fixed by applying the argument on Lemma \ref{L:make_incidence}
inductively, down on the levels $H_{i+1}-\mathring{H}_{i}$. Something similar is sketched in Remark \ref{rm:parallel}.}
for other values of $i$. But that is fine for our purposes: these other surfaces have an intermediate role in the
argument and will eventually be discarded. Still, we assume by general position that every $\mathcal{S}^i$ consists of embedded surfaces $S^{i}_{j}\subseteq H_0$.

Let
\[
\mathcal{F}=\bigcup_{0\leq i\leq h_0} \mathcal{S}^i=\{ S^{h_0}_1,\dots, S^{h_0}_l, S^{h_0-1}_{1},\dots
S^{h_0-1}_{k_1},\dots\dots, S^{0}_{1}, \dots, S^0_{k_{h_0}}\},
\]
where the $S^j_i$'s are ordered lexicographically: decreasing with $j$ and increasing with $i$.
\[
\begin{array}{cccc}
S^{h_0}_1,\dots, S^{h_0}_l &\in & \mathcal{S}^{h_0}=\mathcal{E}&=f^{h_0}(\mathcal{D}) \\
S^{h_0-1}_{1},\dots S^{h_0-1}_{k_1} &\in & \mathcal{S}^{h_0-1}&=f^{h_0-1}(\mathcal{D}) \\
S^{h_0-2}_{1},\dots S^{h_0-2}_{k_2} &\in& \mathcal{S}^{h_0-2}&=f^{h_0-2}(\mathcal{D}) \\
\vdots & &\vdots &\\
S^0_{1},\dots,S^0_{k_{h_0}} &\in& \mathcal{S}^{0}&=\mathcal{D}.
\end{array}
\]
Even if each $\mathcal{S}^i$ is a ``system of surfaces'' the whole $\mathcal{F}$ is not because distinct
$\mathcal{S}^i$'s may intersect. We shall consider the incidence matrix of this collection of embedded surfaces with
respect to $f$. It is clear from the construction that $f(\mathcal{S}^i)\cap H_0=\mathcal{S}^{i+1}$.

Recall our immediate goal is to consider the incidence matrix $W$ of $\mathcal{F}$ (as a collection of properly embedded
surfaces in $H_0$) with respect to $f$. Rows and columns of $W$ follow the lexicographic order of the $S^j_i$'s.
Therefore rows and columns of $W$ also follow the order of the $\mathcal{S}^j$'s decreasingly: the first ones correspond
to
$\mathcal{S}^{h_0}=\mathcal{E}$, then to $\mathcal{S}^{h_0-1}$ and so fourth, so the last ones correspond to
$\mathcal{S}^0=\mathcal{D}$. We then obtain $W$ as a ``block'' matrix. Below we illustrate how such a matrix looks like,
in a case where $\Delta$ is an $(-4)$--tightening disc with height $h_0=4$ and
$\mathcal{F}=\mathcal{S}^4\cup\mathcal{S}^3\cup\mathcal{S}^2\cup\mathcal{S}^1\cup\mathcal{S}^0$.
\begin{equation}\label{eq:matrix_3}
 W=\left[
\begin{array}{ccc|c|ccc|c|ccc}
  &   & &   &   &   & & & & & \\
  & M & & 0 &   &   & & & & & \\
  &   & &   &   &   & & & & & \\ \cline{1-3}
  & * & &   &   & 0 & & & & & \\ \cline{1-4}
  &   & &   &   &   & &0& & & \\
  & * & & * &   &   & & & & & \\
  &   & &   &   &   & & & &0& \\ \cline{1-7}
  & * & & 0 &   & * & & & & & \\ \cline{1-3} \cline{4-11}
  &   & &   &   &   & & & & & \\
  & * & & 0 &   & 0 & &*& &0& \\
  &   & &   &   &   & & & & &
\end{array}\right]
\end{equation}
The first ``block'' consists of the incidence matrix of the original system $\mathcal{E}=\mathcal{S}^{h_0}$.
Since $\mathcal{E}$ is admissible (i.e. $f(\mathcal{E})\cap H_0\subseteq\mathcal{E}$) to the right of $M$ (as a
submatrix of $N$)
all entries are zero.

The next surfaces $S\in\{S^{h_0-1}_{1},\dots,S^{h_0-1}_{k_1}\}= \mathcal{S}^{h_0-1}$ have the property that
$f(S)\cap H_0\subseteq\mathcal{E}$,
which accounts to the non-zero block (in $W$) just below $M$, and zeros to its right. More generally, if
$S\in\mathcal{S}^{i}$ then $f(S)\cap H_0$ consists of discs in $\mathcal{E}$ or surfaces in $\mathcal{S}^{i+1}$, so $W$
has zeros
all above the diagonal to the right of $M$, where this diagonal may be regarded as a ``block diagonal'': with the
exception
of the submatrix $M$, all entries whose row and column correspond to surfaces in the same $\mathcal{S}^i$ are zeros.
In fact non-zero entries on a row corresponding to a disc in $\mathcal{S}^i$ appear only at the first columns
(those corresponding to the original disc system $\mathcal{E}$) or at those corresponding to discs in the following
$\mathcal{S}^{i+1}$. We obtained:

\begin{lemma}\label{L:block-matrix}
The incidence matrix $W$ of $\mathcal{F}=\bigcup_{0\leq i\leq h_0} \mathcal{S}^i$ with respect to $f$
has the form described in (\ref{eq:matrix_3}) above.
\end{lemma}

Now focus on $\mathcal{F}=\bigcup_{0\leq i\leq h_0} \mathcal{S}^i$ as a whole, meaning that for a while it will not
matter much in which $\mathcal{S}^i$ a surface of $\mathcal{F}$ is contained. Enumerate
$\mathcal{F}=\{ F_1,\dots, F_{m}\}=\{ S^{h_0}_1,\dots, S^{h_0}_l, \dots\dots, S^{0}_{1}, \dots, S^0_{k_{h_0}}\},$
following the lexicographic order (hence $m=k+l_1+\cdots+ k_{h_0}$).
Let $u\in\mathbb{R}^{m}$ be the vector:
\[
 u_i=\inter{F_i}{\mathcal{H}_{i_0}}{\hat\nu^{i_0}}\geq0.
\]
Since $W$ is the incidence matrix of $f$, whose growth rate is $\lambda$, $u$ is a $\lambda$--eigenvector of $W$:
\begin{equation}\label{eq:block-matrix}
 Wu=\lambda u.
\end{equation}

Recall the nested $\mathcal{D}_k\subseteq \mathcal{D}_{k+1}$ obtained from the original tightening disc $\Delta$
($\mathcal{D}_0=\{\Delta\}$) and that $\mathcal{S}^{0}=\mathcal{D}=\mathcal{D}_K$, where $\mathcal{D}_K$ is stabilized.
It also
contains $\Delta$, which we assume to be $\Delta=F_{m}\in\mathcal{S}^0=\mathcal{D}$, the last disc of $\mathcal{F}$.
Therefore the last row and column of $W$ correspond to $\Delta$ (the column, naturally, being zero). Also:
\[
 u_{m}=\inter{\Delta}{\mathcal{H}_{i_0}}{\hat\nu^{i_0}}.
\]
Recall we are assuming that $F_1=E_1\in\mathcal{E}$ is the disc of the original system which is parallel to the tightening
disc $\Delta$.
Therefore
\begin{equation}\label{eq:tightening}
u_{m}<u_1=\inter{F_1}{\mathcal{H}_{i_0}}{\hat\nu^{i_0}}.
\end{equation}

The general idea of the argument is to perform a sequence of operations on  $W$ related to $\hat f=\phi\circ
f\circ\varphi$ so that, after passing to a submatrix, we obtain $\hat M$. The growth is controlled in the process.

First obtain $T=T_0$ from $W$ by copying its last line over the first.
\begin{equation}
T= T_0=\left[
\begin{array}{ccc|c|ccc|c|ccc}
           & *    &             &   &   &   & &*& &  &  \\ \cline{1-3} \cline{8-8}
 e_{(2,1)} & \cdots & e_{(2,k)} & 0 &   &   & & & & & \\
 \vdots    & \ddots & \vdots    &   &   & 0 & & & & & \\
 e_{(k,1)} & \cdots & e_{(k,k)} &   &   &   & &0& & & \\ \cline{1-3}
  & * & &   &   &   & & & &0& \\ \cline{1-4}
  &   & &   &   &   & & & & & \\
  & * & & * &   &   & & & & & \\
  &   & &   &   &   & & & & & \\ \cline{1-7}
  & * & & 0 &   & * & & & & & \\ \cline{1-3} \cline{4-11}
      &       &       &   &   &   & & &       &   &       \\
      &   *   &       & 0 &   & 0 & &*&       & 0 &       \\
      &       &       &   &   &   & & &       &   &
\end{array}\right]
\end{equation}
Recall the vector $u\in\mathbb{R}^{m}$:
\[
 u_i=\inter{F_i}{\mathcal{H}_{i_0}}{\hat \nu}\geq 0,
\]
previously built, where $\mathcal{F}=\{ F_1,\dots, F_{m}\}$ is the collection of surfaces.

\begin{lemma}\label{L:first_move}
In addition to $ T_0 u\leq \lambda u$ the following holds:
for every $1\leq j\leq m$ there exists a power $p$ such that $(T_0^p u)_j< \lambda^p u_j$.
\end{lemma}
\begin{proof}
Let $T=T_0$. The first assertion ($T u\leq \lambda u$) is clear: from (\ref{eq:tightening}), for $j=1$
\[
(T u)_1=(Wu)_{m}=\lambda u_{m}<\lambda u_1.
\]
(In fact this already proves the second assertion for $j=1$, in which case $p=1$).

For $2\leq j\leq m$ the $j$--row of $T$ equals that of $W$, therefore from (\ref{eq:block-matrix}) above $(T
u)_j=Wu_j=\lambda u_j$, proving the first assertion.

We now prove the second assertion. The general idea is that the reduction $\lambda^{-1}(T u)_1<u_1$
``propagates'' to all entries by taking powers. Geometrically it comes from the fact that every disc
of $\mathcal{F}$, when iterated by increasing powers of $f$, eventually intersects $H_0$ in a parallel copy
of $D_1=E_1\in\mathcal{E}$ (the disc corresponding to the first row of $W$). In the combinatorics of the matrix this is
related to the following: 1- the incidence matrix $M$ of the original system $\mathcal{E}$ is irreducible and 2- the
blocks of $W$ are as described in Lemma \ref{L:block-matrix}.

We make it precise, beginning with the first entries of $u$, those corresponding to discs of the original
system $\mathcal{E}$, i.e. $1\leq j\leq l$. We have already proved it for $j=1$.

Assume $2\leq j\leq l$. An induction argument (on $q$) like the one in Lemma \ref{L:propagation_matrix} proves the
following.
Fix $q\geq 1$. If for any $1\leq p< q$ the $j$--row of $T^p$ equals that of $W^p$ and its first column is zero
then the $j$--row of $T^q$ equals that of $W^q$. But the upper left block of $W^q$ is $M^q$, while the upper right
consists of zeros.
For $1\leq j\leq l$ recall the integer function
$q(j)$ of Lemma \ref{L:propagation_matrix} corresponding to the irreducible matrix $M$ and fix $q=q(j)$.
By using Lemma \ref{L:propagation_matrix} we obtain that for any $2\leq j\leq l$, the $j$--row
of $T^q$ equals that of $W^q$, which consists of the $j$--row of $M^q$ followed by zeros.
Put another way, $(T^q u)_j=(W^q u)_j=\lambda u_j$. Moreover the first entry of the $j$--row
of $T^q$ is not zero.

Now let $p=q+1$ and consider $(T^p u)_j=(T^q\cdot T u)_j$. We have already proved that $T u\leq\lambda u$
and $(T u)_1<\lambda u_1$ therefore
\[
 (T^p u)_j=(T^q\cdot T u)_j<\lambda^p u_j.
\]
This proves the lemma for every $1\leq j\leq l$. It still remains to check the case $l<j\leq m$.

Let $F=F_j\in\mathcal{F}$ be the corresponding surface. Then $F\in\mathcal{S}^k=f^k(\mathcal{D})\cap H_0$ for some
$0\leq k< h_0$.
Let $r=h_0-k$. Therefore $f^r(F)$ intersects $H_0$ in discs of $\mathcal{S}^{h_0}=\mathcal{E}$. In terms of the matrix
that means
that $T^r$ has a non-zero entry on the $j$--row and some $i$--column for $1\leq i\leq l$. We already proved that there
exists power $p$ such that $(T^p u)_i<\lambda^p u_i$. Therefore:
\[
(T^{r+p} u)_j=(T^r\cdot T^p u)_j<\lambda^{r+p} u_j.
\]
\end{proof}

\begin{obs}
We can regard $T_0$ as the incidence matrix of the collection of surfaces $\mathcal{F}$ with respect to a
$f\circ\varphi$. Indeed $f\circ\varphi(F_1)=f\circ\varphi(E_1)=f(\Delta)=f(F_m)$. The other surfaces are not affected by the introduction of $\varphi$.
\end{obs}

Now change the last rows of $T_0$, those corresponding to the disc system $\mathcal{S}^0=\mathcal{D}$. Recall $\phi$,
the framed isotopy such that $\phi\circ f(\mathcal{D})\cap H_0$ consists of discs of $\mathcal{E}\cup\mathcal{D}$. By
recording the incidence pattern of $\phi\circ f(\mathcal{D})\cap H_0$ we obtain a matrix $T_1$:
\begin{equation}\label{eq:matrix_8}
 T_1=\left[
\begin{array}{ccc|c|ccc|c|ccc}
           & *    &             &   &   &   & &*& &  &  \\ \cline{1-3} \cline{8-8}
 e_{(2,1)} & \cdots & e_{(2,k)} & 0 &   &   & & & & & \\
 \vdots    & \ddots & \vdots    &   &   & 0 & & & & & \\
 e_{(k,1)} & \cdots & e_{(k,k)} &   &   &   & &0& & & \\ \cline{1-3}
  & * & &   &   &   & & & &0& \\ \cline{1-4}
  &   & &   &   &   & & & & & \\
  & * & & * &   &   & & & & & \\
  &   & &   &   &   & & & & & \\ \cline{1-7}
  & * & & 0 &   & * & & & & & \\ \cline{1-3} \cline{4-11}
d_{(k+1,1)} & \cdots & d_{(k+1,k)} &   &   &   & & & d_{(k+1,k+1)} & \dots & d_{(k+1,k+l)} \\
\vdots      & \ddots & \vdots      & 0 &   & 0 & &0& \vdots        &\ddots & \vdots \\
d_{(k+l,1)} & \cdots & d_{(k+l,k)} &   &   &   & & & d_{(k+l,k+1)} & \dots & d_{(k+l,k+l)}
\end{array}\right]
\end{equation}
where the lower non-zero blocks are obtained from (\ref{eq:matrix_2}).

\begin{lemma}\label{L:preserveI}
The matrix $T_1$ and the vector $u\in\mathbb{R}^m$ have the properties that
\begin{enumerate}
    \item $T_1 u\leq\lambda u$,
    \item for any $1\leq j\leq m$ there exists power $p\geq 1$ such that $(T_1^p u)_j<\lambda u_j$.
\end{enumerate}
\end{lemma}
\begin{proof}
Note that by Lemma \ref{L:first_move} and Lemma \ref{L:subinvariance} it suffices to prove that $T_1 u\leq T_0 u$.

If $j$ does not correspond to a disc of $\mathcal{S}^0=\mathcal{D}$ then the $j$--row of $T_1$ equals that of $T_0$
and hence $(T_1 u)_j=(T_0u)_j$.

Therefore assume $j$ corresponds to a disc of $\mathcal{D}$. In this case we go back to the inequality $\bar M v\leq
\lambda v$ of Lemma \ref{L:early_subinvariance} (see its statement for the definition of $v$). Note that the last entries
of $v\in\mathbb{R}^{k+l}$ and $u\in\mathbb{R}^m$, $m>k+l$, coincide, those corresponding to discs of $\mathcal{D}$. Also
for the first entries, those corresponding to $\mathcal{E}$. By also recalling the relation between the lower blocks of
(\ref{eq:matrix_2}) and (\ref{eq:matrix_8}) it is clear that if $v_i$ correspond to the same disc as $u_j$ then $T_1
u_j=\bar M v_i$. But $\bar M v_i\leq\lambda v_i=\lambda u_j$ completing the proof.
\end{proof}

Now obtain $T_2$ by copying the last line of $T_1$ over its first.

\begin{equation}
 T_2=\left[
\begin{array}{ccc|c|ccc|c|ccc}
 d_{(k+l,1)} & \cdots & d_{(k+l,k)} &   &   &   & & & d_{(k+l,k+1)} & \dots & d_{(k+l,k+l)}\\ \cline{1-3} \cline{9-11}
 e_{(2,1)} & \cdots & e_{(2,k)} & 0 &   &   & & & & & \\
 \vdots    & \ddots & \vdots    &   &   & 0 & & & & & \\
 e_{(k,1)} & \cdots & e_{(k,k)} &   &   &   & &0& & & \\ \cline{1-3}
  & * & &   &   &   & & & &0& \\ \cline{1-4}
  &   & &   &   &   & & & & & \\
  & * & & * &   &   & & & & & \\
  &   & &   &   &   & & & & & \\ \cline{1-7}
  & * & & 0 &   & * & & & & & \\ \cline{1-3} \cline{4-11}
d_{(k+1,1)} & \cdots & d_{(k+1,k)} &   &   &   & & & d_{(k+1,k+1)} & \dots & d_{(k+1,k+l)} \\
\vdots      & \ddots & \vdots      & 0 &   & 0 & &0& \vdots        &\ddots & \vdots \\
d_{(k+l,1)} & \cdots & d_{(k+l,k)} &   &   &   & & & d_{(k+l,k+1)} & \dots & d_{(k+l,k+l)}
\end{array}\right]
\end{equation}

\begin{lemma}\label{L:preserveII}
The matrix $T_2$ and the vector $u\in\mathbb{R}^m$ have the properties that
\begin{enumerate}
    \item $T_2 u\leq\lambda u$,
    \item for any $1\leq j\leq m$ there exists power $p\geq 1$ such that $(T_2^p u)_j<\lambda u_j$.
\end{enumerate}
\end{lemma}
\begin{proof}
As in Lemma \ref{L:preserveI} (and using it) it suffices to show that $T_2 u\leq T_1 u$. For $j>1$ it is evident that
$(T_2 u)_j=(T_1 u)_j$.

Let $j=1$. Recall that the first and last lines of the matrices correspond to the discs $F_1=E_1\in\mathcal{E}$,
$F_m=\Delta\in\mathcal{D}$. Moreover
\[
(T_1 u)_1=\inter{f(\Delta)}{\mathcal{H}_{i_0}}{\hat\nu^{i_0}}
\]
and
\[
(T_2 u)_1=\inter{\phi\circ f(\Delta)}{\mathcal{H}_{i_0}}{\hat\nu^{i_0}}.
\]
Direct application of Proposition \ref{P:stabilized} yields $(T_2 u)_1\leq(T_1 u)_1$.
\end{proof}

Erase all lines and columns of $T_2$ but those corresponding to
$\mathcal{S}^{h_0}\cup\mathcal{S}^0=\mathcal{E}\cup\mathcal{D}$, obtaining $T_3$. Clearly $T_3=\hat M$, see
(\ref{eq:matrix_2}). Being a submatrix of $T_2$, from the vector $u\in\mathbb{R}^m$
we extract the corresponding vector $v\in\mathbb{R}^{k+l}$ (which is the same considered in Lemma
\ref{L:early_subinvariance} and Claim \ref{Cl:claim}).

\begin{proof}[Proof of Claim \ref{Cl:claim}]
Considering $\hat M=T_3$ is a submatrix of $T_2$, with $v\in\mathbb{R}^{k+l}$ the vector extracted from
$u\in\mathbb{R}^m$, apply Lemma \ref{L:preserveII} and Corollary \ref{C:submatrix}.
\end{proof}

\bibliographystyle{amsalpha}
\bibliography{reference_leo}

\providecommand{\bysame}{\leavevmode\hbox to3em{\hrulefill}\thinspace}
\providecommand{\MR}{\relax\ifhmode\unskip\space\fi MR }
\providecommand{\MRhref}[2]{%
  \href{http://www.ams.org/mathscinet-getitem?mr=#1}{#2}
}
\providecommand{\href}[2]{#2}
\begin{thebibliography}{Bon83}

\bibitem[BH92]{BH:Tracks}
M.~Bestvina and M.~Handel, \emph{Train tracks and automorphisms of free
  groups}, Ann. of Math. \textbf{135} (1992), 1--51.

\bibitem[BH95]{BH:Surfaces}
\bysame, \emph{Train-tracks for surface homeomorphisms}, Topology \textbf{34}
  (1995), no.~1, 109--140.

\bibitem[Bon83]{FB:CompressionBody}
F.~Bonahon, \emph{Cobordism of automorphisms of surfaces}, Ann. Sci. \'Ecole
  Norm. Sup. (4) \textbf{16} (1983), no.~2, 237--270.

\bibitem[Car06]{LC:tightness}
Leonardo~N. Carvalho, \emph{Tightness and efficiency of irreducible
  automorphisms of handlebodies}, Geom. Topol. \textbf{10} (2006), 57--95.

\bibitem[CO05]{LCUO:Classification}
Leonardo~N. Carvalho and Ulrich Oertel, \emph{A classification of automorphisms
  of compact $3$-manifolds}, pre-print, available ArXiv math.GT/0510610, 2005.

\bibitem[Oer02]{UO:Autos}
U.~Oertel, \emph{Automorphisms of $3$-dimensional handlebodies}, Topology
  \textbf{41} (2002), 363--410.

\bibitem[Oer07]{UO:MCG_compression}
Ulrich Oertel, \emph{Mapping class groups of compression bodies and
  $3$--manifolds}, pre-print, arXiv:math/0607444, 2007.

\bibitem[Sen73]{ES:73}
E.~Seneta, \emph{Non-negative matrices}, Halsted Press, 1973.

\end{thebibliography}

\end{document}